\documentclass[11pt]{article}
\usepackage{a4wide}
\usepackage[a4paper, margin=2.3cm]{geometry}
\usepackage{amsmath,amssymb,amsthm, comment, enumitem, array}
\usepackage{color}
\usepackage{parskip}
\usepackage[pagebackref=true]{hyperref}
\usepackage{parskip}
\usepackage{setspace}
 \usepackage{graphicx}
\usepackage{mathtools}
\usepackage[thinc]{esdiff}
\usepackage[euler]{textgreek}
\usepackage[capitalise, nameinlink]{cleveref}
\usepackage{float}
\usepackage[normalem]{ulem}

\newtheorem{thmlit}{Theorem}

\crefname{equation}{}{}
\crefname{figure}{{\sc Figure}}{{\sc Figure}}
\crefname{subsection}{Subsection}{Subsections}

\usepackage{amsthm}

\makeatletter
\newtheorem*{rep@theorem}{\rep@title}
\newcommand{\newreptheorem}[2]{%
\newenvironment{rep#1}[1]{%
 \def\rep@title{#2 \ref{##1}}%
 \begin{rep@theorem}}%
 {\end{rep@theorem}}}
\makeatother

\newreptheorem{theorem}{Theorem}

\newtheorem{theorem}{Theorem}[section]

\newtheorem{remark}[theorem]{Remark}
\newtheorem{corollary}[theorem]{Corollary}
\newtheorem{lemma}[theorem]{Lemma}
\newtheorem{question}[theorem]{Question}

\newtheorem*{definition*}{Definition}

\newtheorem{construction}[theorem]{Construction}

\hypersetup{
    colorlinks=true,
    linkcolor = blue,
    citecolor=magenta,
    urlcolor=cyan,
}

\def\F{\mathcal{F}}

\def\Fbb{\mathbb{F}}

\def\G{\mathcal{G}}

\def \F{{\mathbb F}}

\def\T{\mathcal{T}}
\def\xx{\mathbf{x}}

\def\l({\left(}
\def\r){\right)}
\def\lv{\left\vert}
\def\rv{\right\vert}
\def\lV{\left\Vert}
\def\rV{\right\Vert}
\def\lo{\left\{}
\def\ro{\right\}}
\def\lu{\left[}
\def\ru{\right]}
\def\lf{\left\lfloor}
\def\rf{\right\rfloor}

\def\rar{\rightarrow}

\def\Tca{\mathcal{T}}
\def\bfx{\mathbf{x}}

\def\yy{\mathbf{y}}

\def\oo{\mathbf{0}}

\begin{document}

\title{Restricted projections in positive characteristic via Fourier extension and restriction estimates}

\author{%
  \begin{tabular}{c}
    % first row: three names, evenly spaced
    \makebox[\linewidth][c]{%
      Le Quang Ham~~~~~~~~~~~~Do Trong Hoang~~~~~~~~~~~~Le Quang Hung%
    }\\[1ex]
    % second row: two names, centered
    \makebox[\linewidth][c]{%
      Doowon Koh~~~~~~~~~~~~Thang Pham%
    }
  \end{tabular}
}
\date{}
\maketitle

\begin{abstract}
Let $d\ge3$ and $\mathbb{F}_q^{\,d}$ be the $d$-dimensional vector space over a finite field of order $q$, where $q$ is an odd prime power. Let $X_\pi$ be the set of lines through the origin intersecting the slice $\pi\cap S^{d-1}$, where $\pi=\{x_d=\lambda\}$ and $S^{d-1}=\{x:\|x\|=1\}$. For $E\subset\mathbb{F}_q^{\,d}$ and $N\ge1$, we study the exceptional sets
\[
\T_1(X_\pi,E,N)=\bigl\{V\in X_\pi:\ |\pi_V(E)|\le N\bigr\},\qquad
\T_2(X_\pi,E,N)=\bigl\{V\in X_\pi:\ |\pi_{V^\perp}(E)|\le N\bigr\},
\]
with their respective natural ranges of $N$. 
Using discrete Fourier analysis together with restriction/extension estimates for cone and sphere-type quadrics over finite fields, 
we obtain sharp upper bounds (up to constant factors) for $\lvert \T_1\rvert$ and $\lvert \T_2\rvert$, with separate analyses for the cases $\lambda \in \{0, \pm 1\}$. The bounds exhibit arithmetic-geometric dichotomies absent in the full Grassmannian: the quadratic character of $\lambda^{2}-1$ and the parity of $d$ determine the size of the exceptional sets. As an application, when $|E|\ge q$, there exists a positive proportion of elements $\mathbf{y}\in \pi\cap S^{d-1}$ such that the pinned dot-product sets $\{\mathbf{y}\cdot \mathbf{x}\colon \mathbf{x}\in E\}$ have cardinality $\Omega(q)$. We further study analogous families arising from the spheres of radii $0$ and $-1$, and, by combining the results, recover the known estimates for projections over the full Grassmannian, complementing a result of Chen (2018).
\end{abstract}



\tableofcontents
\section{Introduction}
For $d\ge 3$, let $\mathbb F_q^d$ be a $d$-dimensional vector space over a finite field with $q$ elements, where $q$ is an odd prime power. For $0<m<d$, the Grassmannian $G(d,m)$ in $\mathbb F_q^d$ is defined as the set of linear subspaces $V$ of $\mathbb F_q^d$ such that $\dim (V)=m.$ For $V\in G(d,m)$ and $E\subset \mathbb F_q^d$, the projection of $E$ onto $V$, denoted by $\pi_V(E)$, is defined by 
$$\pi_V(E) = \{\mathbf{x} +V^{\perp}: (\xx+V^{\perp})\cap E \ne \emptyset, \mathbf{x}\in \mathbb F_q^d\},$$
where $V^{\perp}:=\{\mathbf{x}\in \mathbb{F}_q^d\colon \mathbf{x}\cdot \mathbf{y}=0, ~\forall \mathbf{y}\in V\}$.

For $V\in G(d, m)$ and $E\subset \mathbb{F}_q^d$, it is clear that $|\pi_V(E)|\le \min \{q^m, |E|\}$ since cosets are disjoint. Therefore, given $X\subset G(d, m)$ and $E\subset \mathbb{F}_q^d$, 
it is natural to consider the following two questions. 

\begin{question}\label{qs1}
Given $X\subset G(d, m)$ and $E\subset \mathbb{F}_q^d$ with $|E|> q^m$, what conditions on $X$ and $E$ do we need so that there exists a positive proportion of elements $V\in X$ such that $|\pi_V(E)|\gg q^m?$    
\end{question}

\begin{question}\label{qs2}
Given $X\subset G(d, m)$ and $E\subset \mathbb{F}_q^d$ with $|E|< q^m$, what conditions on $X$ and $E$ do we need so that there exists a positive proportion of elements $V\in X$ such that $|\pi_V(E)|\gg |E|?$   
\end{question}

For $X\subset G(d, m)$ and $E\subset \mathbb F_q^d,$ we define
\begin{align}
 \label{L1}  \mathcal{T}_1(X, E, N)&:=    \left\lbrace V\in X: \#\{\mathbf{x} +V^{\perp}: (\xx+V^{\perp})\cap E \ne \emptyset, x\in \mathbb F_q^d\} \le N \right\rbrace\\
 &=\left\lbrace V\in X: |\pi_V(E)| \le N \right\rbrace;\nonumber
 \end{align}
and
\begin{align} \label{L2}  \mathcal{T}_2(X, E, N)&:=\left\lbrace V\in X: \#\{\mathbf{x} +V: (\xx+V)\cap E \ne \emptyset, \xx\in \mathbb F_q^d\} \le N \right\rbrace\\
 &=\left\lbrace V\in X: |\pi_{V^\perp}(E)| \le N \right\rbrace.\nonumber\end{align}
When $X=G(d, m)$, the topic was initially studied by Chen in \cite{Ch18}, Bright and Gan \cite{BG} then refined Chen's argument to obtain improvements. Recent progress can be found in \cite{LPV} by Lund, Pham, and Vinh for $d=2$, and in \cite{F.R.25} by Fraser and Rakhmonov for Salem sets. These results are of independent interest and also form part of a broader program 
aimed at developing a general framework for improving the current bounds in the 
Erd\H{o}s--Falconer distance problem \cite{CEHIK12, IR, MPPRS22}, inspired by the machinery developed in 
\cite{Du3, alex-fal} for the continuous setting.

For geometric intuition, we include the following pictures. 

\begin{figure}[H]
  \centering
  \includegraphics[width=0.6\textwidth]{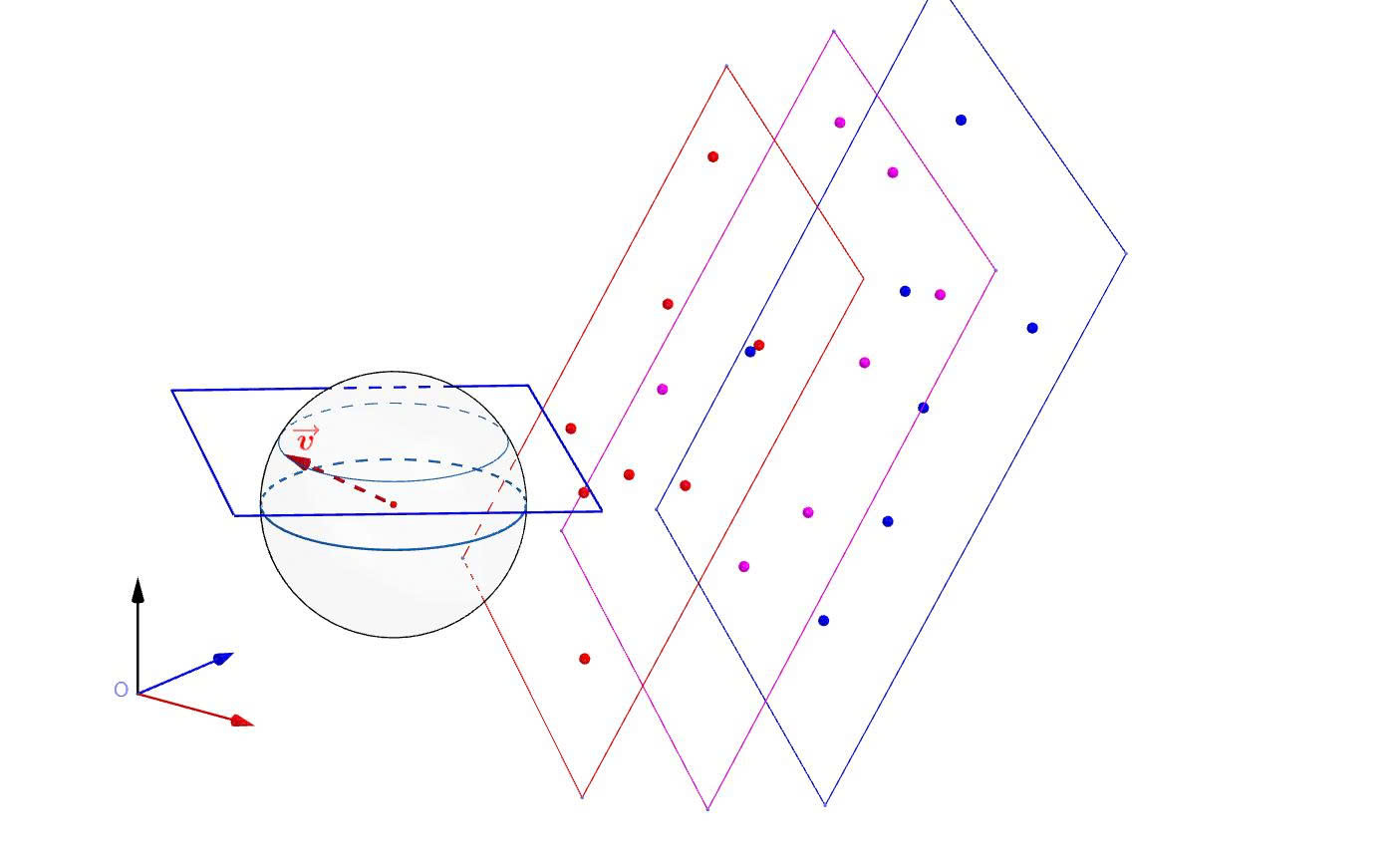}
  \caption{Let $V$ be the line passing through the origin in the direction $\vec{v}$.
  This picture presents the geometric meaning of $\pi_{V}(E)$, i.e. the number of cosets (planes) of $\mathbf{x}+V^\perp$ with non-empty intersection with $E$. All points in each coset will be mapped to a unique point on $V$.}  
  \label{fig:fixed1}
\end{figure}

\begin{figure}[H]
  \centering
  \includegraphics[width=0.6\textwidth]{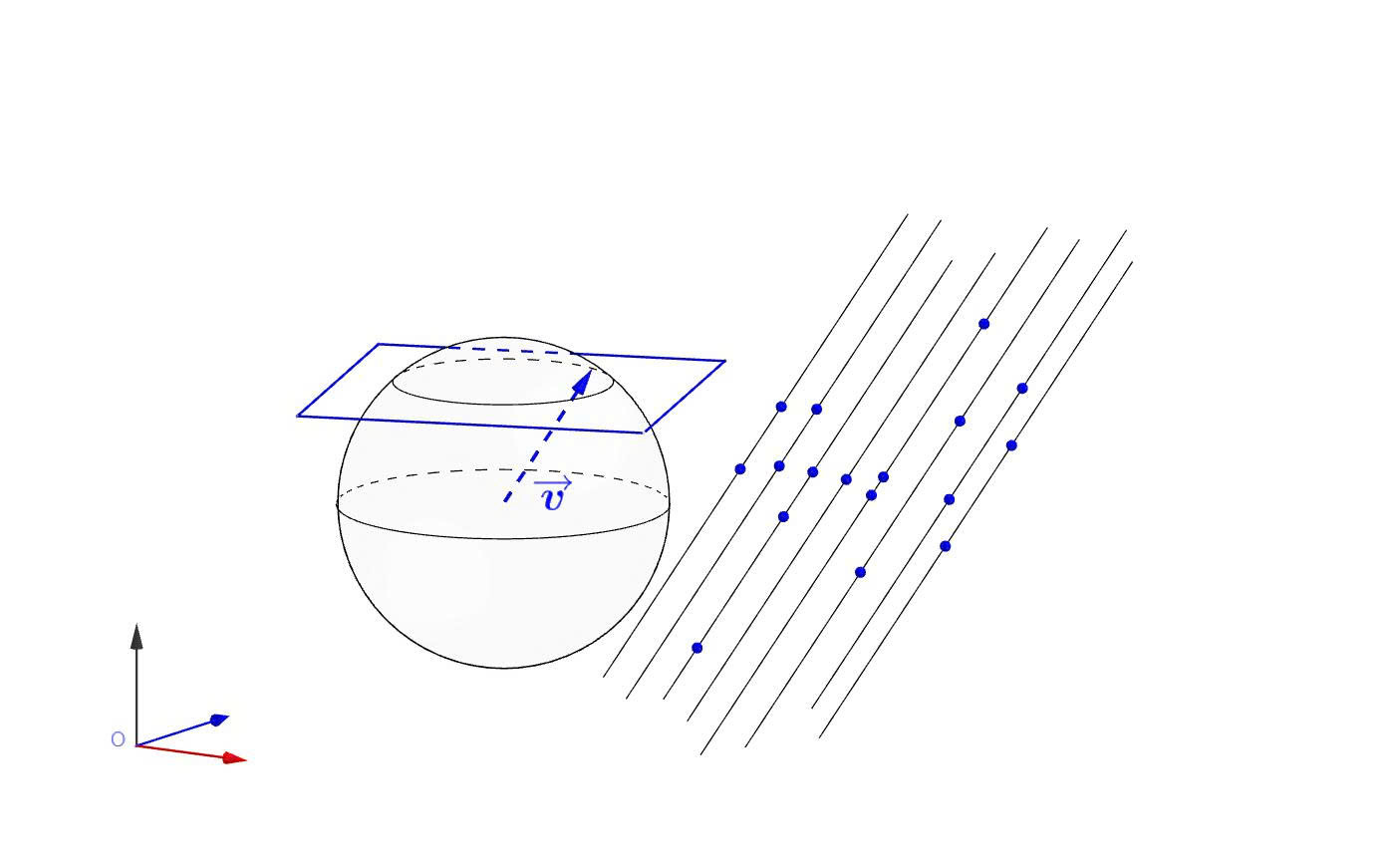}
  \caption{Let $V$ be the line passing through the origin in the direction $\vec{v}$. This picture presents the geometric meaning of $\pi_{V^\perp}(E)$, i.e. the number of cosets (lines) of $\mathbf{x}+\left(V^\perp\right)^\perp=\mathbf{x}+V$ with non-empty intersection with $E$. All points in each coset will be mapped to a unique point on $V^\perp$.} 
  \label{fig:fixed}
\end{figure}

\subsection{Focus of this paper}
For an integer $d\ge 3$, let $S^{d-1}$ be the sphere centered at the origin of radius $1$ in $\mathbb{F}_q^d$, i.e. 
\[S^{d-1}:=\{\xx = \l( x_1 ,x_2 ,\dots ,x_d \r) \in \mathbb{F}_q^d\colon ||\xx||=x_1^2+\cdots+x_d^2=1\}.\]
Let $\lambda \in \mathbb F_q$, let $\pi$ be a hyperplane in $\mathbb{F}_q^d$ defined by the equation $x_d = \lambda$, i.e.
$$\pi = \{\xx = (x_1, x_2, \dots, x_d) \in \mathbb{F}_q^d : x_d = \lambda\}.$$ 
For a set $A\subset \mathbb F_q^d,$  we define $[A]$ as the collection of lines passing through the origin  and  a non-zero element in $A.$ Denote  $X_\pi:=[\pi\cap S^{d-1}]$. It is clear that $X_\pi\subset G(d, 1)$.

The main purpose of this paper is to give complete solutions to Questions \ref{qs1} and \ref{qs2} when $X=X_{\pi}$, namely, we establish sharp upper bounds for $\mathcal{T}_1(X_\pi, E, N)$ and $\mathcal{T}_2(X_\pi, E, N)$. Our approach reveals arithmetic-geometric dichotomies that do not appear in the full Grassmannian: the quadratic character of $\lambda^2-1$, together with the parity of $d$, and congruence conditions on $q$, fundamentally determine the size of exceptional sets, and how they interact to produce distinct projection behaviors. We prove the results are optimal through explicit constructions.
%The finite field setting exhibits new phenomena driven by arithmetic and structural features absent in the Euclidean case. 

In the Euclidean setting, this topic has a rich history which we do not review here. The most relevant works on this topic are \cite{K.O.V.25, L.25} for $\mathcal{T}_1$ and \cite{1} for $\mathcal{T}_2$. Moreover, the methods required in the two settings differ significantly. While Fourier methods or combinatorial/incidence techniques have been used in \cite{BG, Ch18, F.R.25, LPV}, the novelty of our approach is to employ restriction estimates for the cone and sphere-type varieties within an appropriate discrete Fourier framework. Although there is a series of papers in restriction/extension theory, or recently in sharp Fourier restriction in the literature, for example see  \cite{I.K.L.20, KLP22, KPS21, L.15, L.19, MT04}
% \cite{MT04, K.23, KPV.21, I.K.L.20, CKP.22, C.20}, {\color{red}?}
and \cite{BCFSST25, GS24, G.I.25}, respectively,  applications to discrete geometry remain limited, primarily appearing in incidence problems \cite{ KP.22, KLP22}
% \cite{KLP22, KP25},?
and distance questions \cite{CEHIK12, KPV.21,KPS.21}. Our work demonstrates new applications of these analytic tools to projection questions.

As a consequence, we provide a new structural proof of Chen's theorem in \cite{Ch18} via slice decomposition, demonstrating how global projection phenomena emerge from arithmetically distinct local behaviors. This \textit{local-to-global} principle not only explains why Chen's bounds take their specific form but could inspire similar decomposition strategies for other problems in discrete geometry. 

In the first three following  theorems,  we establish upper  bounds for   $\lv \mathcal{T}_1\l( X_\pi, E,N\r) \rv$, corresponding to   the cases $\lambda\in \mathbb{F}_q\setminus\{0, \pm 1\}, \lambda\in \{\pm 1\}$, and $\lambda=0$, respectively. 

\begin{theorem}\label{thm1.3}
    Let  $E\subset \mathbb F_q^d$,  $\lambda\in \mathbb{F}_q\setminus \{0, \pm 1\}$, and $0<N<q$.
     \begin{itemize}
    \item [(a)] If $d\ge 3$ is odd, then 
 \[ \lv \mathcal{T}_1\l( X_\pi, E,N\r) \rv \le \frac{Nq^{d-1}}{(q-N)|E|} + \frac{Nq^{\frac{d-1}{2}}}{q-N}.\]
\item [(b)] If $d\ge 4$ is even and $\eta\left( (-1)^{\frac{d}{2}} (\lambda^2-1) \right)=-1,$ then
\[\lv \mathcal{T}_1\l( X_\pi, E,N\r) \rv \le \frac{Nq^{d-1}}{(q-N)|E|} + \frac{Nq^{\frac{d-2}{2}}}{q-N} .\]

\item [(c)] If $d\ge 4$ is even and $\eta\left( (-1)^{\frac{d}{2}} (\lambda^2-1) \right)=1,$ then
    \[ \lv \mathcal{T}_1\l(X_\pi, E,N\r) \rv\le \frac{Nq^{d-1}}{(q-N)|E|} + \frac{N \l(q^{\frac{d}{2}} -q^{\frac{d-2}{2}}\r)}{q-N} .\]

\end{itemize}
\end{theorem}

%{\color{blue}Theorem 1.3 (a) is sharp in general. To see this, let $E=\{(x, x^2, x^3)\colon x\in \mathbb{F}_q\}$, then $|E|=q$, and any plane in $\mathbb{F}_q$ contains at most $3$ points from $E$. Thus, with $N=3$, we have $T_1=|X_\pi|\sim q$. Moreover, $\frac{Nq^3}{(q-N)|E|}=\frac{3q^3}{(q-3)q}\sim q$. So, the upper bound of $\frac{Nq^3}{(q-N)|E|}$ cannot be improved to $\frac{Nq^{3-\epsilon}}{(q-N)|E|}$ for all $\epsilon>0$.}

\begin{theorem}\label{p11}  Let $E\subset \mathbb F_q^d$, $\lambda \in \lo \pm 1 \ro$, and $0<N<q$, then we have
\begin{align*}
    \lv \mathcal{T}_1\l( X_\pi, E, N\r) \rv \le \begin{cases}
        \dfrac{Nq^{d-1}}{(q-N)|E|}  + \dfrac{Nq^{\frac{d}{2}}}{q-N} ,& \text{if $d \ge 4$ and even},\\[10pt]
        \dfrac{Nq^{d-1}}{(q-N)|E|}  + \dfrac{Nq^{\frac{d-1}{2}}}{q-N} ,&  \text{if $d \equiv 3 \pmod{4}$ and $q\equiv 3 \pmod 4$},\\[10pt]
        \dfrac{Nq^{d-1}}{(q-N)|E|}  + \dfrac{Nq^{\frac{d+1}{2}}}{q-N} , & \text{otherwise}.
    \end{cases}
\end{align*}
 
\end{theorem}

\begin{theorem}\label{thm1.55}
    Let $E\subset \mathbb F_q^d$,  $\lambda=0$, and $0<N<q$. For $\mathbf{x}\in \mathbb{F}_q^{d-1}$, define
    \[
f(\mathbf{x}) := \#\{t \in \mathbb{F}_q : (\mathbf{x},t) \in E\}.
\]
We have 
$$ \lv \mathcal{T}_1\l(X_\pi, E, N\r) \rv \le \frac{2q^{d-1}N}{(q-N)|E|} \cdot \max_{\xx\in \mathbb{F}_q^{d-1}}f(\xx).$$
\end{theorem}

\textbf{Sketch of proofs.}
The proofs of Theorems \ref{thm1.3}, \ref{p11}, and \ref{thm1.55} share a common outline but work with different ambient varieties.
First, by Lemma \ref{Lem3.1KP}, we control $|\mathcal{T}_1(X_\pi,E,N)|$ via a weighted sum of $|\widehat{E}(\xi)|^2$ over the frequency set
\[
K_\pi \;:=\; \bigcup_{V\in X_\pi} (V\setminus\{\mathbf{0}\}) \;\subset\; \F_q^d. ~~(\mbox{union with multiplicity})
\]
For $\lambda\not\in \{0, \pm 1\}$ (Theorem \ref{thm1.3}), this union agrees (up to $\{ \oo\}$) with the cone-type quadric
\[
H_\lambda=\Bigl\{\, \xx\in\F_q^d:\; x_1^2+\cdots+x_{d-1}^2+\tfrac{\lambda^2-1}{\lambda^2}\,x_d^2=0\,\Bigr\},
\]
so we apply the restriction estimates for $H_\lambda$ to bound $\sum_{\xi\in H_\lambda}|\widehat{E}(\xi)|^2$. The final results depend on the arithmetic of $\lambda$, the parity of $d$, and on congruence conditions on $q$. 

The borderline case $\lambda = \pm 1$ (Theorem~\ref{p11}) is degenerate:
the coefficient of $x_d^2$ in the equation of $H_\lambda$ vanishes, so
\[
H_\lambda
= \bigl\{\mathbf{x} \in \mathbb{F}_q^d : x_1^2 + \cdots + x_{d-1}^2 = 0\bigr\}
= \bigl\{x_1^2 + \cdots + x_{d-1}^2 = 0\bigr\} \times \mathbb{F}_q.
\]
In this situation, we no longer use cone-type extension estimates, instead,
we work directly on the quadratic variety
$\{x_1^2 + \cdots + x_{d-1}^2 = 0\}\subset\mathbb{F}_q^{d-1}$ and obtain
the required bounds for
$\sum_{\xi \in K_\pi} |\widehat{E}(\xi)|^2$ by an explicit Fourier-analytic
computation. This yields the estimate stated in Theorem~\ref{p11}.

For $\lambda = 0$ (Theorem~\ref{thm1.55}), we have $\pi = \{x_d = 0\}$. In this case, $K_\pi$ lies in the
hyperplane $\{\xi_d = 0\}$ and is parametrized by directions on $S^{d-2}$.
We then express the weighted sum
\[
\sum_{V \in X_\pi} \sum_{\xi \in V \setminus \{\mathbf{0}\}}
  |\widehat{E}(\xi)|^2
\]
in terms of $\widehat{f}$ on $\mathbb{F}_q^{d-1}$. A direct application of
Plancherel (in dimension $d-1$) then yields the bound stated in
Theorem~\ref{thm1.55}. Using extension estimates for spheres, we also obtain
additional upper bounds for $|\mathcal{T}_1|$ (Theorem~\ref{thm1.55'}), which are of independent interest.

The next theorem is on the size of $\T_2 (X_\pi ,E,N )$.
\begin{theorem}\label{thm:1.5} Let $E\subset \mathbb{F}_q^d$. If $0 <N\le \frac{q^{d-1}}{2}$, then
 \[ |\T_2 (X_\pi ,E,N )|\le \begin{cases}
    \dfrac{4Nq^{d-1}}{|E|}, & \text{if $\lambda =0$},\\[10pt]
    \dfrac{6Nq^{d-2}}{|E|}, & \text{otherwise}.
\end{cases}  \]
\end{theorem}

%The different bounds for $\lambda = 0$ and $\lambda \neq 0$ reflect different geometric configurations: when $\lambda = 0$, we work with the linear hyperplane $\{x_d = 0\}$ (a subspace), while for $\lambda \neq 0$, we work with the affine hyperplane $\{x_d = \lambda\}$ (which does not contain the origin). This distinction leads to different multiplicities in our counting arguments.

% \thang{Why n? not $\xi$?}
\textbf{Sketch of proof.}
The proof of this theorem is geometric/combinatorial. First, by Lemma \ref{Lem3.1KP}, we control $|\mathcal{T}_2(X_\pi,E,N)|$ via a weighted sum of $|\widehat{E}(\xi)|^2$ over the frequency set
\[
K_\pi' \;:=\; \bigcup_{V\in X_\pi} V^\perp \;\subset\; \F_q^d. ~~(\mbox{union with multiplicity})
\]
The contribution of a nonzero frequency $\xi$ is controlled by the multiplicity
\[
M(\xi )\;=\;\#\{\,V\in X_\pi:\; V\subset \xi^\perp\,\}\qquad(\text{equivalently } \xi \in V^\perp).
\]
Passing from lines to their unit representatives on the sphere turns this into a point count on suitable slices, and Lemma \ref{lm2.3} provides uniform bounds, showing that only few directions from $X_\pi$ can lie in any single $\xi^\perp$.
This yields the stated bounds for $\mathcal{T}_2(X_\pi,E,N)$, with a one-power-of-$q$ gap between $\lambda =0$ and $\lambda \neq 0$ reflecting the linear versus affine geometry of the ambient slice.

We now present an application. Observe that if 
\[ \lv \Tca_1 \l( X_\pi , E , N \r) \rv< \frac{|X_\pi|}{C},\]
for some positive $C>1$, then there will be at least $\left(1-\frac{1}{C}\right)|X_{\pi}|$  points $\yy \in \l( \pi \cap S^{d-1}\r) $ such that the size of  
$ \lo \mathbf{y}\cdot\mathbf{x} \colon \mathbf{x}\in E \ro$  is of at least $N$. From Theorems \ref{thm1.3} and \ref{p11}, we obtain the following optimal theorem on the pinned dot-product problem. 

\begin{theorem}\label{thm1.7}
    For sufficiently large $q$, let $E \subset \mathbb{F}_q^d$, $d\ge 3$, and $\lambda \in \mathbb{F}_q \setminus \{ 0\}$. Assume that $|E| \ge q$. Then, except the case that $d=3$, $\lambda \in \{\pm 1\}$, and $q\equiv 1\pmod 4$, 
there are at least $\lf \frac{|X\pi|}{3} \rf$ elements $\yy$ in $\pi \cap S^{d-1}$ such that the set $\{ \mathbf{y} \cdot \mathbf{x} : \mathbf{x} \in E\}$ is of size at least $\frac{q}{10}$.
\end{theorem}

% \thang{It is better to be precise here: for example $|E|\ge cq$ then the number of pinned ... is at least $c'q$. Find values of c and c'.}

% \thang{It is better to use a different notation rather than $\mathbf{v}$.}

\subsection{Sharpness of results} The results, in general, are optimal. 
\begin{enumerate}
    \item It follows from Theorem \ref{thm1.3} ($d\ge 5$) and Theorem \ref{p11} ($d\ge 6$) that 
\[ \lv \Tca_1 \l( X_\pi , E , \frac{q+1}{2} \r) \rv< |X_\pi|,\]
provided that $|E|\ge 10q$. In smaller dimensions, we have \[ \lv \Tca_1 \l( X_\pi , E , \frac{q+1}{5} \r) \rv< |X_\pi|,\]
provided that $|E|\ge q$. Construction \ref{fcon} shows that these results are essentially sharp. In particular, we construct a set $E\in \mathbb{F}_q^d$ with $|E|=q$ and \[ \lv \Tca_1 \l( X_\pi , E , \frac{q+1}{2} \r) \rv=|X_\pi|.\]
Hence, Theorem \ref{thm1.3} and Theorem \ref{p11} are sharp up to constant factors. 
\item The upper bound of Theorem \ref{thm1.55} is attainable (Construction \ref{cons1.5}).
\item Theorem \ref{thm:1.5} with $\lambda \neq 0$ is optimal, namely, when $N \ll \min\{q^{d-1}, |E|\}$, the size of $\T_2(X_\pi, E, N)$ is significantly smaller than the size of $X_\pi$.  It is not possible to extend the range of $N$ to $N \ll q^{d-1+\epsilon}$ for any $\epsilon > 0$ (Construction \ref{16c}). In the case $\lambda=0$, the upper bound of $\frac{Nq^{d-1}}{|E|}$ cannot be improved to $\frac{Nq^{d-1-\epsilon}}{|E|}$ for any $\epsilon>0$ (Construction \ref{c16.2}).
\end{enumerate}

\subsection{Comparison with results in the literature -- the local-to-global principle}
We first recall results on $\Tca_1 (G(d,1) , E ,N )$ and $\Tca_2 (G(d,1) , E ,N )$ due to Chen in \cite{Ch18} and due to Bright and Gan in \cite{BG}. 

%Chen’s and Bright–Gan’s theorems address $X=G(d,m)$ and thus cannot capture improvements that occur when directions are \emph{constrained} (e.g.\ to an algebraic curve on a quadric). 

\begin{thmlit}[Chen, \cite{Ch18}]\label{cor:0.4}
    Let $E\subset \Fbb_q^d$. Then

    (a) for any $N < \frac{1}{2}|E|$, 
    \[ \lv \Tca_1 (G(d,1) , E ,N ) \rv \le 4q^{d-2} N ;   \]
    (b) for any $0< N<q$,
    \[ \lv \Tca_1 (G(d,1) , E ,N) \rv \le 2 \frac{q^{d}N}{(q-N)|E|}. \] 
\end{thmlit}

\begin{thmlit}[Chen, \cite{Ch18}]\label{cor:0.5}
    Let $E \subset \Fbb_q^d$. Then,

    (a) for any $N < \frac{|E|}{2}$,
    \[ \lv \Tca_2 (G(d,1), E , N) \rv \le 4N; \]
    (b) for any $ 0 < N < q^{d-1}$
    \[ \lv \Tca_2 (G(d,1),E, N) \rv \le  \frac{2q^{2d-2}N}{(q^{d-1}-N)|E|} . \]
\end{thmlit}

\begin{thmlit}[Bright--Gan, \cite{BG}]\label{cor:0.6}
    Let $E \subset \Fbb_q^d$ with $|E|=q^a$ $(d-2 < a < d)$. Then, for $1< N < q^\frac{a+2-d}{2} $, we have
    \[ \lv \Tca_1 (G(d,1),E,N) \rv \le C_{d,a,N} \cdot \log q \cdot \frac{q^{d-1}N^2}{|E|^2}.\]
\end{thmlit}

\begin{thmlit}[Bright--Gan, \cite{BG}]\label{cor:0.7}
    Let $E \subset \Fbb_q^d$ with $|E|=q^a$. Then, for $1 < N < q^{\frac{a+d-2}{2}}$, we have
    \[ \lv \Tca_2 (G(d,1) , E ,N) \rv \le C_{d,a,N} \cdot \log q \cdot \max \lo \frac{q^{d-1}N^2}{|E|^2} , q^{d-2} \ro.\]
\end{thmlit}

 A direct computation reveals that each of the two   estimates $\lv \Tca_1 (G(d,1) , E ,N) \rv <|X_{\pi}|$ and $\lv \Tca_2 (G(d,1) , E ,N) \rv<|X_{\pi}|$ leads  either to a trivial conclusion or to a narrow range of  $N$. Therefore, they are not relevant when considering $X=X_\pi$. 

 As we mentioned earlier about the local-to-global principle, i.e. putting our results together recovers Chen's theorem. To see this, let $S^{d-1}_{-1}$ and $S^{d-1}_0$ be the spheres centered at the origin of radius $-1$ and $0$ in $\Fbb_q^d$, respectively, i.e.
\[ S_{-1}^{d-1}:= \lo \bfx \in \Fbb_q^d \colon \lV \bfx \rV =-1 \ro , \quad S_0^{d-1}:= \lo \bfx  \colon \lV \bfx \rV =0 \ro . \]
Define $Y_\pi := \lu\pi \cap S^{d-1}_{-1}\ru$ and $Z _\pi = \lu \pi\cap S_0^{d-1}  \ru$.

In the finite field setting, the spheres of radii \(1\) (a square) and \(-1\)
(a nonsquare when \(q\equiv 3 \pmod 4\)) are not equivalent under scalings. This contrasts with the Euclidean case,
where all nonzero radii are equivalent under dilation. This explains why we have to treat the two cases separately. By using the same approach, we obtain similar results (Theorems 
\ref{thm1.4}, \ref{p11i}, \ref{pi02}, \ref{p10}, and \ref{thm:1.5''}) for $Y_\pi$ and $Z_\pi$.

Using the cover $G(d,1)\subset \bigcup_{\lambda\in \mathbb{F}_q} \left(X_\pi\cup Y_\pi\cup Z_\pi\right)$\footnote{To simplify notation, we suppress the $\lambda$-dependence and do not write
$X_{\pi}^{\lambda}$, $Y_{\pi}^{\lambda}$, or $Z_{\pi}^{\lambda}$. When needed,
the dependence on $\lambda$ will be clear from context.}, which is true when $q\equiv 3\pmod 4$, and an observation that for some specific values of $\lambda$, say $\lambda=0$, the contribution to $\T_1$ and $\T_2$ is at most $O(q^{d-2})$, one has
    \[ \lv \Tca_1 (G(d,1) , E ,N) \rv \ll  \frac{q^{d}N}{(q-N)|E|}+q^{d-2}, \]
    for any $0< N<q$. This is essentially the same as Chen's result, since $q^{d-2}=o(|G(d,1)|)$. For any $ 0 < N < q^{d-1}$, the same strategy implies
    \[ \lv \Tca_2 (G(d,1),E, N) \rv \ll  \frac{q^{2d-2}N}{(q^{d-1}-N)|E|} . \]
Based on our sharpness examples, the averaging shows that Chen’s global bounds are essentially optimal on average, i.e. the inequalities
\[ \lv \Tca_1 (G(d,1) , E ,N) \rv \ll \frac{q^{d}N}{(q-N)|E|} ~\mbox{and}~\lv \Tca_2 (G(d,1),E, N) \rv \ll  \frac{q^{2d-2}N}{(q^{d-1}-N)|E|}\]
cannot be improved to 
\[ \lv \Tca_1 (G(d,1) , E ,N) \rv \ll \frac{q^{d-\epsilon}N}{(q-N)|E|}~\mbox{and}~\lv \Tca_2 (G(d,1),E, N) \rv \ll  \frac{q^{2d-2-\epsilon}N}{(q^{d-1}-N)|E|}\]
for any $\epsilon>0$. This provides an affirmative answer to Chen's question on the sharpness when looking at $G(d, 1)$.

Note that the cover
\[
G(d,1)\subset \bigcup_{\lambda\in \mathbb{F}_q} \bigl(X_\pi \cup Y_\pi \cup Z_\pi\bigr)
\]
fails when $q\equiv 1 \pmod 4$, since in that case $-1$ is a square and the
spheres $S^{d-1}_1$ and $S^{d-1}_{-1}$ represent the same square class in
$\mathbb{F}_q^*$. In this situation, we can fix a nonsquare
$c\in\mathbb{F}_q^*$ and replace $S^{d-1}_{-1}$ by the sphere
$S^{d-1}_c$. Running the same argument yields the same
conclusions for $\mathcal{T}_1$ and $\mathcal{T}_2$. We do not spell out the details here.

\textbf{Notation:} Throughout the paper, we write $A \ll B$ or $B \gg A$ if there exists a
constant $C>0$, independent of $q$, such that $A \le C B$. When we need
to indicate dependence on auxiliary parameters, we write $A \ll_{d,m} B$,
etc. We write $A \sim B$ if $A \ll B$ and $B \ll A$.

\section{Preliminary results}
  Let  $d \in \mathbb N$   and $f:\mathbb F_q^d \rightarrow \mathbb C$ be a complex valued function. The Fourier transform of $f$ at $ \xi \in \mathbb F^d_q$, denoted by $\widehat{f}$, is defined by
\[\widehat{f}(\xi):=\sum_{\xx\in \mathbb F_q^d} \chi (-\xi \cdot \xx)f(\xx),\]
where $\chi$ is a nontrivial additive character of $\mathbb F_q$, and the dot product 
$$\xx\cdot 
\xi \equiv x_1 \xi_1+ x_2\xi_2+ \cdots +x_d \xi_d \ ({\rm mod} \ q).$$
Recall the following  Plancherel identity 
\[ \sum_{\xi \in \mathbb F_q^d} |\widehat{f}(\xi )|^2 = q^d\sum_{\xx\in \mathbb F_q^d}|f(\xx)|^2. \]

 The next lemma is an an extension of \cite[Lemmas 2.3]{Ch18}, from prime fields to arbitrary finite fields $\mathbb{F}_q$.  
 \begin{lemma}[Corollary 4.3, \cite{F.R.25}]\label{Lem3.1KP}
    For $W\in G(d, m)$, let $\{\xx_{W_j}+W\}_{j=1}^{q^{d-m}}$ be the set of disjoint translations of $W \subset \mathbb F_q^d.$  Then, for every $E\subset \mathbb F_q^d$, we have 
    \[\sum_{j=1}^{q^{d-m}}|E\cap (\xx_{W_j}+W)|^2=q^{-(d-m)}\sum_{\xi\in W^\perp}|\widehat{E}(\xi)|^2.\]
\end{lemma}

 Let $\G$ denote the standard Gauss sum
$$\G:=\sum_{t\in \mathbb F_q^*} \eta(t) \chi(t),$$
where $\eta$ is the quadratic character of $\Fbb_q^\ast$, i.e. a group homomorphism
defined by $\eta(t)=1$ if $t$ is a square, $\eta(0)=0$ (conventional), and $-1$ otherwise. 
Recall that  the absolute value of the Gauss sum $\G$  is $\sqrt{q}$, and 
$\G^2= \eta(-1) q.$ The next lemma records the classical explicit form of the quadratic Gauss sum.
\begin{lemma}[Theorem 5.15, \cite{LN97}]\label{ExplicitGauss}
Let $\mathbb F_q$ be a finite field of order $q=p^{r}$, where $p$ is an odd prime and $r \in {\mathbb N}.$
Then we have
$$\mathcal{G}=\left\{\begin{array}{ll}  {(-1)}^{r-1} q^{\frac{1}{2}} \quad &\mbox{if} \quad p \equiv 1 \pmod 4 \\
                  {(-1)}^{r-1} i^r q^{\frac{1}{2}}  \quad &\mbox{if} \quad p\equiv 3 \pmod 4.\end{array}\right.$$
\end{lemma}
This lemma tells us that if $d\equiv 2\pmod 4$ and $q\equiv 3\pmod 4$ then $\mathcal{G}^d=-q^{\frac{d}{2}}$. In several upper-bound arguments, the negative sign lets us discard the would-be leading term, giving a sharper estimate.

The next lemma will also be used in some places. 
\begin{lemma}[Lemma  2.3, \cite{KLP22}]\label{lem:2.5} For $a\in  \mathbb{F}_q^*$ and  $b\in  \mathbb{F}_q$, we have 
 \begin{equation*}
\sum_{s \in \mathbb F_q} \chi (a s^2+bs) =
\eta(a)\cdot \mathcal{G} \cdot\chi  \left(\frac{b^2}{-4a}\right).  
    \end{equation*}     
\end{lemma}

\section{Tools from restriction and extension theory}
Given a function $f\colon \mathbb{F}_q^d\to \mathbb{C}$, we note that its Fourier transform is a complex-valued function defined on the dual space. Since the dual space of $\mathbb{F}_q^d$ is isomorphic to $\mathbb{F}_q^d$, we use the same notation $\mathbb{F}_q^d$ for the dual space. 

For a variety $V\subset \mathbb{F}_q^d$, the $L^b\to L^a$ Fourier extension problem for $V$ asks us to determine all exponents $1\le a, b\le \infty$ such that the following inequality holds 
\begin{equation}\label{extension}
\left(\sum_{\xi \in \mathbb{F}_q^d}\left\vert \frac{1}{|V|}\sum_{\xx\in V}\chi( \xi \cdot \xx)f(\xx)\right\vert^a\right)^{\frac{1}{a}}\le C\left(\frac{1}{|V|}\sum_{\xx\in V}|f(\xx)|^b\right)^{\frac{1}{b}},\end{equation}
for all functions $f$ on $V$, where the positive constant $C$ does not depend on $q$.  By duality, the inequality (\ref{extension}) is equivalent with the following Fourier restriction estimate 
\begin{equation}\label{restriction} 
\left(\frac{1}{|V|}\sum_{\xx\in V}|\widehat{g}(\xx)|^{b'}\right)^{\frac{1}{b'}}\le C\left(\sum_{\xi \in \mathbb{F}_q^d}|g(\xi)|^{a'} \right)^{\frac{1}{a'}},\end{equation}
for any function $g$ on the dual space $\mathbb{F}_q^d$, where $1/a+1/a'=1, ~1/b+1/b'=1$. In this section, by the notation $R_V^{*}(b\to a)\ll 1$, we mean that the inequality (\ref{extension}) holds.

\subsection{Restriction estimates for a cone--type variety}
Let $\lambda\in \mathbb F_q\backslash \{0,\pm 1\}$. We denote   $$ H_\lambda:=\left\{\mathbf{x}=(x_1, \ldots, x_d)\in \mathbb F_q^d:  x_1^2+ \cdots + x_{d-1}^2 +\frac{\lambda^2-1}{\lambda^2} x_d^2 =0\right\}.$$
Note that $H_\lambda= \{s\mathbf{x}: s\in \mathbb F_q,  \mathbf{x}\in X_\pi\}$  which is a cone-type variety.

%Let $\eta$ denote the quadratic character of $\mathbb F_q,$ that is   
% $$\eta(t)= \begin{cases} 1 &\text{ if }   t \text{  is a square number in }  \mathbb F_q^* \\
%  -1&\text{ otherwise.}
%  \end{cases}$$ 

  For  $\lambda \ne \pm 1,$    the cardinality of $ X_\pi$  is the same as  that of $\pi \cap S^{d-1}.$
 Thus, we see that
\begin{eqnarray*}
   |X_\pi|&=& \lv \pi \cap S^{d-1}\rv= \sum_{\substack{ (x_1, \ldots, x_{d})\in \mathbb F_q^d:\ x_d=\lambda,\\ x_1^2+ \cdots+ x_{d-1}^2 =1-\lambda^2  }} 1.
\end{eqnarray*}
  Namely,  $|X_\pi|$ is the number of the solutions of the equation $x_1^2+ \cdots+ x_{d-1}^2 =1-\lambda^2$, which is well-known as
follows (see \cite[Theorem A.1]{BHIPR17})

\begin{lemma}  \label{Lem1.6}   For $\lambda \in \Fbb_q$, let $X_\pi$ be as defined before.
% $\lambda \in \mathbb F_q\setminus \{-1, 1\}$.   
\begin{itemize}
\item [(i)] If $d\ge 2$ is even, then we have
\[|X_\pi| = \begin{cases}
    q^{d-2} +q^{\frac{d-2}{2}} \eta\left( (-1)^{\frac{d}{2}} (\lambda^2-1)\right) , & \text{if $\lambda \notin \lo -1,1 \ro$},\\
    q^{d-2} , & \text{if $\lambda \in \lo -1 ,1 \ro$}.
\end{cases} \]
\item [(ii)] If $d\ge 3$ is odd, then  
\[ |X_\pi| = \begin{cases}
    q^{d-2} -q^{\frac{d-3}{2}} (\eta(-1))^{\frac{d-1}{2}} , & \text{if $\lambda \notin \lo -1,1 \ro$},\\
    q^{d-2} +q^{\frac{d-1}{2}} (\eta(-1))^{\frac{d-1}{2}}  -q^{\frac{d-3}{2}} (\eta(-1))^{\frac{d-1}{2}}, & \text{if $\lambda \in \lo -1,1 \ro$}.
\end{cases} \]
\end{itemize}
\end{lemma}

We set  \[R_\lambda(E):= \sum_{\xi\in H_\lambda} |\widehat{E}(\xi)|^2.\] Utilizing discrete Fourier analysis and estimates of Gauss sums, we establish upper bounds for $R_\lambda(E)$.

\begin{lemma} \label{Lem1.9}Let $E\subset \mathbb F_q^d.$ Then the following statements hold.

\begin{itemize} 
\item [(i)] If $d\ge 2$ is even, and $\eta\left( (-1))^{\frac{d}{2}} (\lambda^2-1) \right)=-1,$ then
$$ R_\lambda(E) \le q^{d-1} |E| + q^{\frac{d-2}{2}} |E|^2.$$

\item [(ii)] If $d\ge 2$ is even, and $\eta\left( (-1))^{\frac{d}{2}} (\lambda^2-1) \right)=1,$ then
$$ R_\lambda(E) \le q^{d-1} |E| + q^{\frac{d}{2}} |E|^2-q^{\frac{d-2}{2}} |E|^2.$$
\item [(iii)] If $d\ge 3$ is odd, then 
$$ R_\lambda(E) \le q^{d-1} |E| + q^{\frac{d-1}{2}}|E|^2.$$
\end{itemize}
\end{lemma}

\begin{proof}
Set $\varepsilon= \frac{\lambda^2-1}{\lambda^2}.$ Then, by the definitions of the Fourier transform and the variety $H_\lambda,$   it follows that

\begin{align*} R_\lambda(E) =&\sum_{\xi\in \mathbb F_q^d: \xi^2_1+\cdots+ \xi_{d-1}^2+ \varepsilon \xi_d^2=0}  \sum_{\alpha, \beta\in E} \chi( (\alpha-\beta)\cdot \xi)\\
=& \sum_{\xi\in \mathbb F_q^d}  \left(q^{-1} \sum_{t\in \mathbb F_q} 
 \chi(t(\xi_1^2+ \cdots + \xi_{d-1}^2 +\varepsilon \xi_d^2)) \right) \sum_{\alpha, \beta\in E} \chi( (\alpha-\beta)\cdot \xi),
 \end{align*}
 where the orthogonality  relation of $\chi$ was also applied.
Separately computing $t=0$ and $t\ne 0$ in the sum over $t\in  \mathbb F_q$,   we obtain that
$$ R_\lambda(E) = q^{d-1} |E| +  q^{-1}\sum_{\alpha, \beta \in E}  \sum_{t\ne 0}   \sum_{\xi\in \mathbb F_q^d}   \chi(t(\xi_1^2+ \cdots + \xi_{d-1}^2 +\varepsilon \xi_d^2)) \chi( (\alpha-\beta)\cdot \xi).$$

Computing the sum in $\xi \in \mathbb F_q^d$ by invoking the Gauss sum estimates,  we get
\begin{equation}\label{FFE} R_\lambda(E) =q^{d-1} |E| + q^{-1} \eta(\varepsilon) \G^d  \sum_{\alpha, \beta \in E} \left( \sum_{t\ne 0} \eta^d(t)  \chi\left ( \frac{(\alpha_d-\beta_d)^2 -\varepsilon ||\alpha'-\beta'||}{4\varepsilon t}   \right) \right),\end{equation}
% \begin{equation}\label{FFE} R_\lambda(E) =q^{d-1} |E| + q^{-1} \eta(\varepsilon) G^d  \sum_{\alpha, \beta \in E} \left( \sum_{t\ne 0} \eta^d(t)  \chi\left ( \frac{(\alpha_d-\beta_d)^2 -\varepsilon ||\alpha'-\beta'||}{4\varepsilon} \frac{1}{t}  \right) \right),\end{equation}
where we write $\alpha'$ for $(\alpha_1, \ldots, \alpha_{d-1}) \in \mathbb F_q^{d-1}.$ 

\textbf{Case 1:}  Assume that $d\ge 3$  is an odd integer.
Then  the sum over $t\ne 0$ in \eqref{FFE} is related to the Gauss sum and its absolute value is $\sqrt{q}.$  Since $|\eta|=1$ we  therefore obtain 
$R_\lambda(E) \le  q^{d-1}|E| +  q^{-1} q^{\frac{d}{2}} q^{\frac{1}{2}} |E|^2,$
which completes the proof of  Lemma \ref{Lem1.9} (iii).\\

\textbf{Case 2:}  Assume that $d\ge 4$ is even.
Since $\eta^d\equiv 1,$   it follows from \eqref{FFE} that
\begin{align*} R_\lambda(E) =& q^{d-1} |E| + q^{-1} \eta(\varepsilon)\G^d  \sum_{\alpha, \beta \in E} \left( \sum_{t\ne 0}  \chi\left ( \frac{(\alpha_d-\beta_d)^2 -\varepsilon ||\alpha'-\beta'||}{4\varepsilon} \cdot  \frac{1}{t}  \right) \right)\displaybreak[3]\\ 
=& q^{d-1} |E| + q^{-1} \eta(\varepsilon) \G^d \sum_{\substack{\alpha, \beta \in E: \\ (\alpha_d-\beta_d)^2 -\varepsilon ||\alpha'-\beta'||=0}}  (q-1) + q^{-1} \eta(\varepsilon) \G^d \sum_{\substack{\alpha, \beta \in E:\\ (\alpha_d-\beta_d)^2 -\varepsilon ||\alpha'-\beta'||\ne 0}}  (-1)\\
=& q^{d-1} |E| + \eta(\varepsilon) \G^d \sum_{\substack{\alpha, \beta \in E:\\ (\alpha_d-\beta_d)^2 -\varepsilon ||\alpha'-\beta'||= 0}}  1-q^{-1} \eta(\varepsilon) \G^d |E|^2.\end{align*}   
Since $\varepsilon= \frac{\lambda^2-1}{\lambda^2}$ and $\G^2=\eta(-1) q,$ 
 
\begin{equation}\label{Eq4K} R_\lambda(E)= q^{d-1} |E| + \eta(\lambda^2-1) (\eta(-1))^{\frac{d}{2}} q^{\frac{d}{2}} \sum_{\substack{\alpha, \beta \in E: \\ (\alpha_d-\beta_d)^2 -\varepsilon ||\alpha'-\beta'||= 0}}  1-\eta(\lambda^2-1) (\eta(-1))^{\frac{d}{2}} q^{\frac{d-2}{2}}|E|^2.\end{equation}
 
\textbf{Case 2.1:} Assume that $\eta\left( (-1))^{\frac{d}{2}} (\lambda^2-1) \right)=-1.$
It follows from \eqref{Eq4K} that
 
$$R_\lambda(E)= q^{d-1} |E| - q^{\frac{d}{2}} \sum_{\substack{\alpha, \beta \in E:\\  (\alpha_d-\beta_d)^2 -\varepsilon ||\alpha'-\beta'||= 0}}  1+ q^{\frac{d-2}{2}}|E|^2.$$
 
Since the second term above is negative,   Lemma \ref{Lem1.9} (i) follows.\\

\textbf{Case 2.2:}  Assume that $\eta\left( (-1))^{\frac{d}{2}} (\lambda^2-1) \right)=1.$ Then  we see from \eqref{Eq4K} that
 
\begin{align*} R_\lambda(E)&= q^{d-1} |E| + q^{\frac{d}{2}} \sum_{\substack{\alpha, \beta \in E:\\ (\alpha_d-\beta_d)^2 -\varepsilon ||\alpha'-\beta'||= 0}}  1- q^{\frac{d-2}{2}}|E|^2\\
& \le q^{d-1} |E| + q^{\frac{d}{2}} |E|^2- q^{\frac{d-2}{2}}|E|^2.\end{align*} 
This completes the proof of  Lemma \ref{Lem1.9} (ii). 
\end{proof}

\subsection{Extension estimates for spheres}

%For $\xx = \l( x_1 ,x_2 , \ldots ,x_d\r) \in \F_q^d$, denote by $Q(\xx)$ a homogeneous polynomial in $\Fbb_q \lu x_1 , x_2 , \ldots ,x_d \ru $ of degree $2$. Since $char (\F_q ) > 2$, $Q(\xx) $ can be expressed as 
%\[ Q(\xx) = \xx A \xx^T, \]
%where $A = \l( a_{ij} \r)_{d\times d}, a_{ij}\in \Fbb_q,\,  \forall i,j$ is a symmetric matrix. If $A$ is invertible, $Q(\xx)$ is said that a non-degenerate quadratic form over $\Fbb_q$. 
For each $j \in \F_q^\ast $, let $S^{d-1}_j$ be the sphere of radius $j$ in $\F^d_q$, defined by
\[ S^{d-1}_j = \{ \xx =(x_1,x_2,\ldots ,x_d) \in \F^d_q \colon \sum_{k=1}^d x_k^2 = j\}.\]
The following results are taken from \cite{I.D.08}. 

% Let $S\subset \Fbb_q^d$ be an algebraic variety in $\Fbb_q^d$. We denote by $d\sigma$ normalized surface measure on $S$ defined by the relation 
% \[ \widehat{fd\sigma}( \bfy ) = \frac{1}{\lv S\rv} \sum_{\bfx \in S} \chi (-\bf x \cdot \yy) f(\xx) . \]

%For $1 \le a,b \le \infty $ let $R^\ast (b\to a)$ be the best constant such that the extension estimate 
%\[ \l( \sum_{\xx \in \Fbb_q^{d-1} }\lv \frac{1}{\lv S^{d-1}_j\rv} \sum_{\yy \in S^{d-1}_j} \chi (\yy \cdot \xx) f(\yy) \rv^{a}  \r)^{\frac{1}{a}} \ll \l( \frac{1}{\lv S^{d-1}_j \rv} \sum_{\yy \in S^{d-1}_j} \lv f(\yy ) \rv^b \r)^{\frac{1}{b}}, \]
%holds for all functions $f: S^{d-1}_j \rar \Cbb $.

% \begin{theorem}[Theorem 1, \cite{I.D.08}]\label{lem6.1}
%     For $j \in \F_q^\ast $, $d \ge 2$, $s \ge \frac{2d+2}{d-1}$, let $f: \F_q^d \rar \Cbb$ be a function, then we have 
%     \[  \l( \sum_{\xx \in \Fbb_q^{d-1} }\lv \frac{1}{\lv S^{d-1}_j\rv} \sum_{\yy \in S^{d-1}_j} \chi (\yy \cdot \xx) f(\yy) \rv^{s}  \r)^{\frac{1}{s}} \ll \l( \frac{1}{\lv S^{d-1}_j \rv} \sum_{\yy \in S^{d-1}_j} \lv f(\yy ) \rv^2 \r)^{\frac{1}{2}}.  \]
% \end{theorem}

\begin{theorem}[Theorem 1, \cite{I.D.08}]\label{lem6.1}
    For $j \in \F_q^\ast $, $d \ge 2$, $a \ge \frac{2d+2}{d-1}$, then  
    \[  R_{S_j^{d-1}}^\ast (2 \rar a ) \ll 1.\]
\end{theorem}

\begin{theorem}[Theorem 2, \cite{I.D.08}]\label{lem6.2}
    For $j \in \F_q^\ast $, $d \ge 2$, then  
    \[ R_{S_j^{d-1}}^\ast (2\rar 4) \ll 1. \]
\end{theorem}

\section{Bounding $\T_1(X, E, N)$ for $\lambda \ne \{0, \pm 1\}$ (Theorem \ref{thm1.3})}
%\subsection{Results in $d$-dimensions}
Let $E\subset \mathbb F_q^d$ and $X\subset G(d, m).$   
\begin{theorem}\label{lem1K} Let $\T_1 \l(X, E,N\r)$ be defined as in \eqref{L1}.   Then we have
$$  \left\vert \mathcal{T}_1\l(X, E, N\r)\right\vert \le \frac{N}{\l( q^m-N\r)|E|^2}  \sum_{V\in \T_1(X, E,N)} \sum_{\xi \in V\setminus \{0\}} |\widehat{E}(\xi)|^2,$$ 
where $\widehat{E}(\xi)= \sum\limits_{\mathbf{x}\in E} \chi(-\mathbf{x}\cdot \xi)$ for $\xi\in \mathbb F_q^d.$
\end{theorem}
\begin{proof}
    For $V \in G(d,m)$, let $\{ \xx_{V^\perp_j} +V^\perp \}_{j=1}^{q^m}$ be the set of disjoint translations of $V^\perp \subset \mathbb F_q^d$. From definition of $\mathcal{T}_1(X, E, N)$ and the Cauchy-Schwarz inequality, we have 
    \begin{align*}
        |E|^2  = \left( \sum_{j=1}^{q^m} |E\cap (\xx_{V_j^\perp}+V^\perp)| \right)^2 & \le |\pi_V (E)|\left( \sum_{j=1}^{q^m} \left\vert E\cap (\xx_{V_j^\perp}+V^\perp) \right\vert^2 \right) \\
        & \le N \left( \sum_{j=1}^{q^m} \left\vert E\cap (\xx_{V_j^\perp}+V^\perp) \right\vert^2 \right) ,
    \end{align*}
    for all $V\in \T_1 \l( X,E,N\r)$.
    By this inequality and Lemma \ref{Lem3.1KP}, we have 
    \begin{align*}
        &\frac{\sum_{\xi \in V\setminus \{ 0\}}|\widehat{E}(\xi)|^2}{|E|^2}  = \frac{\sum_{\xi \in V}|\widehat{E}(\xi )|^2}{|E|^2} -1 = \frac{q^m\left( \sum_{j=1}^{q^m} \left\vert E\cap (x_{V_j^\perp} +V^\perp) \right\vert^2 \right)}{|E|^2} -1 \\
        & \ge \frac{q^m}{N} - 1 = \frac{q^m-N}{N},
    \end{align*}
    for all $V \in \T_1(X,E,N)$.
    Therefore, taking the sum over all $V \in \Tca_1 (X,E,N)$ gives 
    \[ \lv \mathcal{T}_1(X,E,N)\rv \le \frac{N}{\l( q^m-N\r) |E|^2} \sum_{V\in \T_1(X, E,N)} \sum_{\xi \in V\setminus \{0\}} |\widehat{E}(\xi)|^2 
     . \]
     This completes the proof.
\end{proof}
% \subsection{An application for the  variety $X_\pi$}
The next theorem shows that we can bound the size of $  \mathcal{T}_1\l(X_\pi, E, N\r)$ by using restriction estimates for the cone-type variety $H_\lambda$. 
\begin{theorem}\label{P1.4}  Let $E\subset \mathbb F_q^d.$  Then we have
% $$ \lv \mathcal{T}_1\l(X_\pi, E, \frac{q}{2}\r) \rv \le \frac{1}{|E|^2} \sum_{\xi\in H_\lambda} |\widehat{E}(\xi)|^2.$$

$$ \lv \mathcal{T}_1\l(X_\pi, E, N\r) \rv \le \frac{N}{(q-N)|E|^2} \sum_{\xi\in H_\lambda} |\widehat{E}(\xi)|^2.$$

\end{theorem} 

\begin{proof} By Theorem \ref{lem1K}, it follows that
% $$\lv \mathcal{T}_1\l(X_\pi, E, \frac{q}{2}\r)\rv \le \frac{1}{|E|^2}  \sum_{V\in \T_1(X_\pi, E,N)} \sum_{\xi \in V\setminus \{0\}} |\widehat{E}(\xi)|^2,$$ 

$$\lv \mathcal{T}_1\l(X_\pi, E, N\r)\rv \le \frac{N}{(q-N)|E|^2}  \sum_{V\in \T_1(X_\pi, E,N)} \sum_{\xi \in V\setminus \{0\}} |\widehat{E}(\xi)|^2,$$ 

 Since $\mathcal{T}_1\l( X_\pi, E,  N \r) \subset X_\pi=[\Pi_\lambda\cap S^{d-1}],$ 
% {\color{red} (or $H_\lambda$?)}
% $$\lv \mathcal{T}_1\l( X_\pi, E, \frac{q}{2}\r) \rv \le \frac{1}{|E|^2}  \sum_{V\in X_\pi} \sum_{\xi \in V\setminus \{0\}} |\widehat{E}(\xi)|^2.$$

$$ \lv \mathcal{T}_1\l( X_\pi, E, N\r) \rv \le \frac{N}{(q-N)|E|^2}  \sum_{V\in X_\pi} \sum_{\xi \in V\setminus \{0\}} |\widehat{E}(\xi)|^2 .$$

Notice that to complete the proof, it is enough to show that  the two claims  below hold:
\vskip0.5em 
\noindent {\bf Claim 1:} $V\subset H_\lambda$ for all $V\in X_\pi$.
\vskip0.5em 
Indeed, let $V\in X_\pi.$ Then $V$ is a line passing through the origin and a point $\mathbf{x}$ in $\pi \cap S^{d-1}.$  In other words, we have
$V=\{s \xx: s\in \mathbb F_q\}$ where $x_1^2+ \cdots+x_{d-1}^2+\lambda^2-1=0$ and $x_d=\lambda.$
Since $(sx_1)^2+ \cdots+ (s x_{d-1})^2 + \frac{\lambda^2-1}{\lambda^2} (sx_d)^2=0$ for all $s\in \mathbb F_q,$  we see that $V\subset H_\lambda.$ Thus,  the claim  holds true.   
\vskip0.5em 
\noindent {\bf Claim 2:} If $V, W \in X_\pi$ with $V\ne W$,  then  $V\cap W=\{(0, \ldots, 0)\}.$
\vskip0.5em 
Indeed, suppose that $V, W\in X_\pi$ with $V\ne W.$ Then for some $\mathbf{x}=(x_1, \ldots, x_{d-1}, \lambda), ~ \mathbf{y}=(y_1, \ldots, y_{d-1}, \lambda)\in \mathbb F_q^d$ with $\mathbf{x}\ne \mathbf{y}$,   we can write
$$ V=\{ s(x_1, \ldots, x_{d-1}, \lambda): s\in \mathbb F_q\} ,$$
and $$W=\{ t(y_1, \ldots, y_{d-1}, \lambda): t\in \mathbb F_q\}.$$
It is clear that  $\oo \in V\cap W.$ Now, let us assume that $\oo \ne \gamma\in V\cap W.$ Then for some $s, t\in \mathbb F_q^*$   we have
$$s(x_1, \ldots, x_{d-1}, \lambda) = t(y_1, \ldots, y_{d-1}, \lambda).$$
Since $\lambda\ne 0,$  we see that $\mathbf{x}=\mathbf{y},$ which makes a contradiction.  In conclusion,   $V\cap W=\{ \oo \},$   as required. 
\end{proof}

\begin{corollary} \label{cor_Eq2} Let $E \subset \F_q^d$. Then, we have
  % $$\lv \mathcal{T}_1\l(X_\pi, E, \frac{q}{2}\r) \rv \le \frac{q^d}{|E|}.$$  

 $$\lv \mathcal{T}_1\l(X_\pi, E, N\r) \rv \le \frac{Nq^d}{(q-N)|E|}.$$  
\end{corollary}
\begin{proof}
      Since $H_\lambda \subset \mathbb F_q^d,$  it is clear from Theorem \ref{P1.4}  that
% $$ \lv \mathcal{T}_1\l( X_\pi, E, \frac{q}{2}\r) \rv \le \frac{1}{|E|^2} \sum_{\xi\in \mathbb F_q^d} |\widehat{E}(\xi)|^2.$$ 

 $$ \lv \mathcal{T}_1\l( X_\pi, E, N\r) \rv \le \frac{N}{(q-N)|E|^2} \sum_{\xi\in \mathbb F_q^d} |\widehat{E}(\xi)|^2.$$ 

By the Plancherel theorem,    $\sum_{\xi\in \mathbb F_q^d} |\widehat{E}(\xi)|^2=q^d |E|.$ 
\end{proof}

\begin{proof}[Proof of Theorem \ref{thm1.3}]
Combining Theorem \ref{P1.4}  and Lemma \ref{Lem1.9},  we obtain the following facts:

\begin{itemize}
\item [(a)] If $d\ge 3$ is odd, then  
 $\lv \T_1\l( X_\pi, E,N\r) \rv\le \dfrac{Nq^{d-1}}{(q-N)|E|} + \dfrac{Nq^{\frac{d-1}{2}}}{(q-N)}.$ 
 
\item [(b)] If $d\ge 4$ is even, and $\eta\left( (-1))^{\frac{d}{2}} (\lambda^2-1) \right)=-1,$ then
 $ \lv \T_1 \l(X_\pi, E,N\r) \rv \le \dfrac{Nq^{d-1}} {(q-N)|E|} + \dfrac{Nq^{\frac{d-2}{2}}}{q-N} .$ 

\item [(c)] If $d\ge 4$ is even, and $\eta\left( (-1))^{\frac{d}{2}} (\lambda^2-1) \right)=1,$ then  
$$ \lv \T_1 \l( X_\pi, E,N \r) \rv \le \dfrac{Nq^{d-1}}{(q-N)|E|} + \dfrac{N\l( q^{\frac{d}{2}} -q^{\frac{d-2}{2}} \r)}{q-N} .$$ 
\end{itemize}

% Hung, 

% In addition,  it is clear from Corollary \ref{cor_Eq2} that 
% \begin{itemize} \item [(a)]  If $d=3$, then  $\lv \T_1 \l( X_\pi, E,N\r) \rv \le \frac{Nq^d}{(q-N)|E|}.$
% \end{itemize}
These bounds complete the proof.
\end{proof}
\begin{corollary}\label{cor3.6}
    For $\epsilon>0$, let  $E\subset \mathbb F_q^d, 0 <N <q$. Then the number of subspaces  $V\in X_\pi\subset G(d, 1)$  such that  $|\pi_V(E)| >  N$ is more than $(1-\epsilon)|X_\pi|$,
provided that one of the following conditions holds: 
\begin{itemize} \item[(i)] $d\ge 4$ is even,  $\eta\left( (-1))^{\frac{d}{2}} (\lambda^2-1) \right)=-1,$ and  $|E| > \dfrac{Nq^{d-1}}{\epsilon (q-N)q^{d-2}-\epsilon q^{\frac{d}{2}} -(1-\epsilon)Nq^{\frac{d-2}{2}}}.$
\item [(ii)] $d\ge 4$ is even,  $\eta\left( (-1))^{\frac{d}{2}} (\lambda^2-1) \right)=1,$ and  $|E| >  \dfrac{Nq^{d-1}}{\epsilon(q-N)q^{d-2}-(N-\epsilon)q^{\frac{d}{2}}+(1-\epsilon)Nq^{\frac{d-2}{2}}}.$
\item [(iii)]  $d\ge 5$ is odd, and $|E|> \dfrac{Nq^{d-1}}{\epsilon (q-N)q^{d-2} - (N+\epsilon)q^{\frac{d-1}{2}} + \epsilon Nq^{\frac{d-3}{2}}\l( \eta (-1) \r)^{\frac{d-1}{2}}}.$ 

\item [(iv)] $d=3$  and $|E| > \dfrac{Nq^3}{\epsilon (q-N)(q-\eta(-1))}.$
\end{itemize}
\end{corollary}
\begin{proof}

First recall that  $\T_1 \l(X_\pi, E, N\r)=\lo V\in X_\pi: |\pi_V(E)| \le N\ro.$
It is clear from the definition of $\T_1 (X_\pi, E,N)$ that 
the condition  $\lv \T_1\l( X_\pi, E,N\r) \rv < \epsilon|X_\pi|$  implies the conclusion of Corollary \ref{cor3.6}. 

Recall Lemma \ref{Lem1.6}:

\begin{itemize}
\item [$(3.1-1)$] If $d\ge 4$ is even, and $\eta\left( (-1))^{\frac{d}{2}} (\lambda^2-1) \right)=-1,$ then 
$|X_\pi| = q^{d-2} -q^{\frac{d-2}{2}} .$
\item [$(3.1-2)$] If $d\ge 4$ is even, and $\eta\left( (-1))^{\frac{d}{2}} (\lambda^2-1) \right)=1,$ then 
$|X_\pi| = q^{d-2} +q^{\frac{d-2}{2}}.$
 \item [$(3.1-3)$] If $d\ge 3$ is odd, then  
$|X_\pi| = q^{d-2} -q^{\frac{d-3}{2}} (\eta(-1))^{\frac{d-1}{2}}.$ 
\end{itemize}
Combining $(3.1-1)$ and Theorem \ref{thm1.3} (b) imply 
\begin{eqnarray*}
    &&\frac{Nq^{d-1}}{(q-N)|E|}+\frac{Nq^{\frac{d-2}{2}}}{q-N} < \epsilon \l( q^{d-2} - q^{\frac{d-2}{2}} \r) \\
    &\Longleftrightarrow & \frac{Nq^{d-1}}{\epsilon (q-N)q^{d-2} -Nq^{\frac{d-2}{2}}-\epsilon (q-N)q^{\frac{d-2}{2}}} < |E| \\
    &\Longleftrightarrow & \frac{Nq^{d-1}}{\epsilon (q-N)q^{d-2}-\epsilon q^{\frac{d}{2}} -(1-\epsilon)Nq^{\frac{d-2}{2}}} < |E|.
\end{eqnarray*} 
 
Combining $(3.1-2)$ and Theorem \ref{thm1.3} (c) imply 
 \begin{eqnarray*}
    && \frac{Nq^{d-1}}{(q-N)|E|} + \frac{N \l(q^{\frac{d}{2}} -q^{\frac{d-2}{2}}\r)}{q-N} < \epsilon \l(q^{d-2} + q^{\frac{d-2}{2}}\r) \\
    &\Longleftrightarrow & \frac{Nq^{d-1}}{\epsilon(q-N)q^{d-2}-(N-\epsilon)q^{\frac{d}{2}}+(1-\epsilon)Nq^{\frac{d-2}{2}}} < |E|.
\end{eqnarray*}
Combining $(3.1-3)$ and Theorem \ref{thm1.3} (a) imply  
\begin{eqnarray*}
     && \frac{Nq^{d-1}}{(q-N)|E|} + \frac{Nq^{\frac{d-1}{2}}}{q-N} < \epsilon \l( q^{d-2}-q^{\frac{d-3}{2}} \l( \eta (-1) \r)^{\frac{d-1}{2}} \r) \\
     &\Longleftrightarrow & \frac{Nq^{d-1}}{\epsilon (q-N)q^{d-2} - Nq^{\frac{d-1}{2}} - \epsilon(q-N)q^{\frac{d-3}{2}}\l( \eta (-1) \r)^{\frac{d-1}{2}}} < |E| \\
     &\Longleftrightarrow & \frac{Nq^{d-1}}{\epsilon (q-N)q^{d-2} - (N+\epsilon)q^{\frac{d-1}{2}} + \epsilon Nq^{\frac{d-3}{2}}\l( \eta (-1) \r)^{\frac{d-1}{2}}} < |E|.
\end{eqnarray*}
 
Finally, combining $(3.1-3)$ with $d=3$ and Corollary \ref{cor_Eq2} imply  
\begin{eqnarray*}
    && \frac{Nq^3}{(q-N)|E|} < \epsilon \l( q-\eta(-1) \r) \\
    &\Longleftrightarrow  & \frac{Nq^3}{\epsilon (q-N)(q-\eta(-1))} < |E|.
\end{eqnarray*}
 
Hence, the proof of Corollary \ref{cor3.6} is complete.
\end{proof}

\section{Bounding $\T_1(X, E, N)$ for $\lambda=\pm 1$ (Theorem \ref{p11})}
Compared to the proof of Theorem \ref{thm1.3}, we need the following analog of Lemma \ref{Lem1.9}.

\begin{lemma} \label{Lem1.911}Let $E\subset \mathbb F_q^d$ and $\lambda \in \lo -1 ,1 \ro$. Then the following statements hold.

\begin{itemize} 
\item [(i)] If $d\ge 4 $ is  even, then
$$ \sum_{V \in X_\pi} \sum_{\xi \in V \setminus \{ 0\}} \lv \widehat{E}(\xi) \rv^2 \le q^{d-1} |E| + q^{\frac{d}{2}} |E|^2.$$

\item [(ii)] For $d\equiv 3 \pmod 4$ and $q\equiv 3 \pmod 4$, we have
\[ \sum_{V \in X_\pi} \sum_{\xi \in V \setminus \{ 0\}} \lv \widehat{E}(\xi) \rv^2 \le q^{d-1} |E| + q^{\frac{d-1}{2}} |E|^2. \]

\item [(iii)] For other cases, we have
$$ \sum_{V \in X_\pi} \sum_{\xi \in V \setminus \{ 0\}} \lv \widehat{E}(\xi) \rv^2 \le q^{d-1} |E| + q^{\frac{d+1}{2}} |E|^2.$$
\end{itemize}
\end{lemma}

\begin{proof} By applying the orthogonality relation of $\chi$, Lemma \ref{lem:2.5}, one has
\begin{align*}
&\sum_{V \in X_\pi} \sum_{\xi \in V \setminus \{ 0\}} \lv \widehat{E}(\xi) \rv^2 = \sum_{k \in \Fbb_q^\ast} \sum_{\substack{\xi \in \mathbb F_q^d \colon \\ \xi^2_1+\cdots+ \xi_{d-1}^2=0,\\ \xi_d=\lambda}}  \lv \widehat{E}(k\xi) \rv^2  \le 
\sum_{\substack{\xi \in \Fbb_q^d \colon \\ \xi_1^2 +\cdots + \xi_{d-1}^2 =0}} \lv \widehat{E}(\xi)\rv^2 = \sum_{\substack{\xi\in \mathbb F_q^d \colon \\ \xi^2_1+\cdots+ \xi_{d-1}^2=0}}  \sum_{\alpha, \beta\in E} \chi( (\alpha-\beta)\cdot \xi)\\
 & =  \sum_{\xi\in \mathbb F_q^d}  \left(\frac{1}{q} \sum_{t\in \mathbb F_q} 
 \chi(t(\xi_1^2+ \cdots + \xi_{d-1}^2)) \right)  \sum_{\alpha, \beta\in E} \chi( (\alpha-\beta)\cdot \xi) \displaybreak[3]\\
 &=\frac{1}{q}\sum_{\xi_d \in \F_q }\sum_{\alpha, \beta\in E}\chi( \xi_d(\alpha_d-\beta_d))\sum_{\xi^\ast \in \mathbb{F}_q^{d-1}}\sum_{t\in \mathbb F_q}\chi(t||\xi^\ast||)\chi\left(\sum_{i=1}^{d-1}(\alpha_i-\beta_i)\xi_i^\ast \right)\\
& = \frac{1}{q}\sum_{\xi_d \in \F_q}\sum_{\alpha, \beta\in E}\chi( \xi_d (\alpha_d-\beta_d))\sum_{\xi^\ast \in \mathbb{F}_q^{d-1}}\chi\left(\sum_{i=1}^{d-1}(\alpha_i-\beta_i)\xi_i^\ast\right)\\
& \qquad +\frac{1}{q}\sum_{\xi_d\in \F_q }\sum_{\alpha, \beta\in E}\chi( \xi_d(\alpha_d-\beta_d))\sum_{\xi^\ast \in \mathbb{F}_q^{d-1}}\sum_{ t\in \mathbb F_q^\ast}\chi(t||\xi^\ast||)\chi\left(\sum_{i=1}^{d-1}(\alpha_i-\beta_i)\xi_i^\ast \right) \\ 
 &=q^{d-2}\sum_{\xi_d \in \F_q}\sum_{\substack{\alpha, \beta \in E\colon \\ \alpha_i=\beta_i,\forall  i\le d-1}}\chi (\xi_d (\alpha_d-\beta_d))\\
 & \qquad +\frac{1}{q}\G^{d-1}\sum_{\xi_d \in \F_q }\sum_{\alpha, \beta\in E}\chi( \xi_d (\alpha_d-\beta_d))\sum_{t\ne 0}\eta(t)^{d-1}\chi\left(\frac{||\alpha'-\beta'||}{-4t}\right)\\
 & = q^{d-2}\sum_{\xi_d \in \F_q}\sum_{\alpha \in E}1+q^{d-2}\sum_{\xi_d \in \F_q}\sum_{\substack{\alpha, \beta\colon \\ \alpha_i=\beta_i, i\le d-1,\\
 \alpha_d \ne \beta_d}}\chi( \xi_d (\alpha_d-\beta_d)) \\
& \qquad +\frac{1}{q}\G^{d-1}\sum_{\xi_d\in \F_q }\sum_{\alpha, \beta\in E}\chi( \xi_d (\alpha_d-\beta_d))\sum_{t\ne 0}\eta(t)^{d-1}\chi\left(\frac{||\alpha'-\beta'||}{-4t}\right)\\
& = q^{d-1}|E| +\frac{1}{q}\G^{d-1}\sum_{\xi_d \in \F_q }\sum_{\alpha, \beta\in E}\chi( \xi_d (\alpha_d-\beta_d))\sum_{t\ne 0}\eta(t)^{d-1}\chi\left(\frac{||\alpha'-\beta'||}{-4t}\right),
 \end{align*}
 
where $\alpha ' = \l( \alpha_1 , \ldots ,\alpha_{d-1}\r) , \beta ' = \l( \beta_1 ,\ldots ,\beta_{d-1} \r)$.
 
{\bf Case 1:} If $d\ge 4$ is even, then $d-1$ is odd. Then, $\eta(t)^{d-1}=\eta(t) = \eta \l( \frac{1}{t} \r)$, and apply Lemma \ref{ExplicitGauss}, one has
\begin{align*}  
    \sum_{V \in X_\pi} \sum_{\xi \in V \setminus \{ \oo\}} \lv \widehat{E}(\xi) \rv^2 & \le q^{d-1}|E| +\frac{1}{q}\G^{d-1}\sum_{\xi_d \in \F_q }\sum_{\alpha, \beta\in E}\chi( \xi_d (\alpha_d-\beta_d))\sum_{t\ne 0}\eta(t)\chi\left(\frac{||\alpha'-\beta'||}{-4t}\right) \\
    & = q^{d-1}|E| +\frac{1}{q}\G^{d-1}\sum_{\xi_d\in \F_q }\sum_{\alpha, \beta\in E\colon \lV \alpha ' -\beta ' \rV =0  }\chi( \xi_d (\alpha_d-\beta_d))\sum_{t\ne 0}\eta(t) \\
    & \quad +\frac{1}{q}\G^{d-1}\sum_{\xi_d \in \F_q }\sum_{\alpha, \beta\in E\colon \lV \alpha ' -\beta ' \rV \ne 0}\chi( \xi_d (\alpha_d-\beta_d))\sum_{t\ne 0}\eta(t)\chi\left(\frac{||\alpha'-\beta'||}{-4t}\right)\\
& = q^{d-1}|E| +\frac{1}{q}\G^{d} \sum_{\xi_d\in \F_q }\sum_{\alpha, \beta\in E\colon \lV \alpha '- \beta ' \rV \ne 0}\chi( \xi_d (\alpha_d-\beta_d))\eta \l( \frac{-1}{\lV \alpha ' -\beta' \rV} \r)  \\
    & = q^{d-1}|E|+ \frac{1}{q}\G^{d} \sum_{\xi_d \in \F_q }\sum_{\alpha, \beta\in E \colon \lV \alpha ' -\beta ' \rV \ne 0}\chi( \xi_d (\alpha_d-\beta_d))\eta \l( \frac{-1}{\lV \alpha ' -\beta' \rV} \r) \\
    & \le q^{d-1}|E| + \frac{1}{q}\G^{d} \sum_{\xi_d \in \F_q }\sum_{\alpha, \beta\in E \colon \lV \alpha ' -\beta ' \rV \ne 0} 1 \le q^{d-1}|E|+\frac{q^{\frac{d}{2}}}{q}q|E|^2  =q^{d-1}|E|+q^{\frac{d}{2}}|E|^2.
\end{align*}
{\bf Case 2:} If $d\ge 3$ is odd, then $\eta^{d-1}(t)=1$. So 
\begin{align*}
    \sum_{V \in X_\pi} \sum_{\xi \in V \setminus \{ \oo\}} \lv \widehat{E}(\xi) \rv^2 & \le q^{d-1}|E| +\frac{1}{q}\G^{d-1}\sum_{\xi_d\in \F_q}\sum_{\alpha, \beta\in E}\chi( \xi_d (\alpha_d-\beta_d))\sum_{t\ne 0}\chi\left(\frac{||\alpha'-\beta'||}{-4t}\right).
\end{align*}
Note that from Lemma \ref{ExplicitGauss}, if $d\equiv 3\pmod 4$, and $q=p^r$ with $r$ odd, $p\equiv 3\pmod 4$, then $\G^{d-1}=-q^{\frac{d-1}{2}}$, otherwise, we have $\G^{d-1}=q^{\frac{d-1}{2}}$.

{\bf Case 2.1:} If $d,q\equiv 3\pmod 4$ and $r$ is odd, then 
\begin{align*}
    \sum_{V \in X_\pi} \sum_{\xi \in V \setminus \{ \oo\}} \lv \widehat{E}(\xi) \rv^2 &\le q^{d-1}|E|-q^{\frac{d-3}{2}}(q-1)\sum_{\xi_d\in \F_q }\sum_{\substack{\alpha,\beta \colon \\ ||\alpha'-\beta'||=0}}\chi( \xi_d (\alpha_d-\beta_d))
     \\ & \quad 
    +q^{\frac{d-3}{2}}\sum_{\xi_d\in \F_q }\sum_{\substack{\alpha,\beta \colon \\ ||\alpha'-\beta'||\ne 0}}\chi( \xi_d (\alpha_d-\beta_d))\\
    % &=q^{d-1}|E|- q^{\frac{d-1}{2}}\sum_{\xi_d\in \F_q }\sum_{\alpha ,\beta \colon ||\alpha'-\beta'||=0}\chi( \xi_d (\alpha_d-\beta_d)) \\
    % & \quad +q^{\frac{d-3}{2}}\sum_{\xi_d\in \F_q }\sum_{\alpha, \beta\in E}\chi( \xi_d (\alpha_d-\beta_d))\\
    &=q^{d-1}|E|- q^{\frac{d-1}{2}}\sum_{\substack{\alpha ,\beta \colon \\ ||\alpha'-\beta'||=0}}\sum_{\xi_d\in \F_q }\chi( \xi_d (\alpha_d-\beta_d)) 
    % \\ & \quad 
    +q^{\frac{d-3}{2}}\sum_{\alpha, \beta\in E}\sum_{\xi_d\in \F_q }\chi( \xi_d (\alpha_d-\beta_d))\\
    & = q^{d-1}|E|- q^{\frac{d-1}{2}}\sum_{\substack{\alpha ,\beta \colon \lV \alpha' -\beta' \rV =0,\\ \alpha_d =\beta_d}} q - q^{\frac{d-1}{2}}\sum_{\substack{\alpha ,\beta \colon \lV \alpha' -\beta' \rV =0,\\ \alpha_d \ne \beta_d}} 0
     \\ & \quad 
    + q^{\frac{d-3}{2}}\sum_{\alpha , \beta \colon \alpha_d =\beta_d} q  + q^{\frac{d-3}{2}}\sum_{\alpha , \beta \colon \alpha_d \ne \beta_d}0 
    % \\
    % &
    \quad \le q^{d-1}|E|+q^{\frac{d-1}{2}}|E|^2.
\end{align*}
% So, 
% \[\sum_{V \in X_\pi} \sum_{\xi \in V \setminus \{ \oo\}} \lv \widehat{E}(\xi) \rv^2\le q^{d-1}|E|+q^{\frac{d-1}{2}}|E|^2.\]

{\bf Case 2.2:} For other cases, we have 
\begin{align*}
   \sum_{V \in X_\pi} \sum_{\xi \in V \setminus \{ \oo\}} \lv \widehat{E}(\xi) \rv^2 &\le q^{d-1}|E|+q^{\frac{d-3}{2}}(q-1)\sum_{\xi_d\in \F_q }\sum_{\substack{\alpha ,\beta \in E \colon \\ ||\alpha'-\beta'||=0}}\chi(\xi_d (\alpha_d-\beta_d))\\
   & \quad 
    -q^{\frac{d-3}{2}}\sum_{\xi_d\in \F_q}\sum_{\substack{\alpha ,\beta \in E \colon \\ ||\alpha'-\beta'||\ne 0}}\chi(\xi_d (\beta_d-\alpha_d))\\
    &=q^{d-1}|E|+q^{\frac{d-1}{2}}\sum_{\alpha ,\beta \colon ||\alpha'-\beta'||=0}\sum_{\xi_d\in \F_q }\chi(\xi_d (\alpha_d-\beta_d)) 
    % \\
    % & \quad 
    -q^{\frac{d-3}{2}}\sum_{\alpha, \beta\in E}\sum_{\xi_d\in \F_q }\chi(\xi_d (\alpha_d-\beta_d))  \\
    & = q^{d-1}|E| + q^{\frac{d-1}{2}}\sum_{\substack{\alpha , \beta \colon \lV \alpha' -\beta' \rV =0,\\ \alpha_d =\beta_d}} q+ q^{\frac{d-1}{2}}\sum_{\substack{\alpha , \beta \colon \lV \alpha' -\beta' \rV =0,\\ \alpha_d \ne \beta_d}} 0 \\
    & \quad - q^{\frac{d-3}{2}}\sum_{\alpha ,\beta \colon \alpha_d =\beta_d}q- q^{\frac{d-3}{2}}\sum_{\alpha ,\beta \colon \alpha_d \ne \beta_d}0.
\end{align*}
Therefore,
$\sum_{V \in X_\pi} \sum_{\xi \in V \setminus \{ \oo\}} \lv \widehat{E}(\xi) \rv^2 \le q^{d-1}|E|+q^{\frac{d+1}{2}}|E|^2.
$
This completes the proof.
\end{proof}
We now recall and prove Theorem \ref{p11}.
\begin{theorem}\label{p1}  Let $E\subset \mathbb F_q^d$,  then we have
\begin{align*}
    \lv \mathcal{T}_1\l( X_\pi, E, N\r) \rv \le \begin{cases}
        \dfrac{q^{d-1}N}{(q-N)|E|}  + \dfrac{q^{\frac{d}{2}}N}{q-N} ,& \text{if $d \ge 4$ and even},\\[10pt]
        \dfrac{q^{d-1}N}{(q-N)|E|}  + \dfrac{q^{\frac{d-1}{2}}N}{q-N} ,&  \text{if $d,q \equiv 3 \pmod 4$},\\[10pt]
        \dfrac{q^{d-1}N}{(q-N)|E|}  + \dfrac{q^{\frac{d+1}{2}}N}{q-N} , & \text{otherwise}.
    \end{cases}
\end{align*}
 
\end{theorem} 

\begin{proof} By Theorem \ref{lem1K}, it follows that
$$\lv \mathcal{T}_1\l(X_\pi, E, N\r)\rv \le \frac{N}{(q-N)|E|^2}  \sum_{V\in \T_1(X_\pi, E,N)} \sum_{\xi \in V\setminus \{\oo\}} |\widehat{E}(\xi)|^2.$$ 
Since $\mathcal{T}_1 (X_\pi , E ,N) \subset X_\pi$, we are reduced to 
$$ \lv \mathcal{T}_1\l( X_\pi, E, N\r) \rv \le \frac{N}{(q-N)|E|^2}  \sum_{V\in X_\pi} \sum_{\xi \in V\setminus \{\oo\}} |\widehat{E}(\xi)|^2 .$$
Now applying Lemma \ref{Lem1.911},we have 
\begin{align*}
    \lv \mathcal{T}_1\l( X_\pi, E, N\r) \rv \le \begin{cases}
        \dfrac{q^{d-1}N}{(q-N)|E|}  + \dfrac{q^{\frac{d}{2}}N}{q-N} ,& \text{if $d \ge 4$ and even},\\[10pt]
        \dfrac{q^{d-1}N}{(q-N)|E|}  + \dfrac{q^{\frac{d-1}{2}}N}{q-N} ,&  \text{if $d,q \equiv 3 \pmod 4$},\\[10pt]
        \dfrac{q^{d-1}N}{(q-N)|E|}  + \dfrac{q^{\frac{d+1}{2}}N}{q-N} , & \text{otherwise}.
    \end{cases}
\end{align*}
This completes the proof.
\end{proof}

% {\color{blue}
\section{Application on the pinned dot-product problem (Theorem \ref{thm1.7})}

\begin{proof}[Proof of Theorem \ref{thm1.7}]
We begin with two preliminary observations.  

\textbf{Observation 1.} From the definition of $\Tca_1$, if there exists a constant $C > 1$ such that  
\[
\big| \Tca_1 \big( X_\pi ,E ,N \big) \big| < \frac{|X_\pi|}{C},
\]
then at least $\big( 1 - \tfrac{1}{C} \big) |X_\pi|$ points $\yy \in \l( \pi \cap S^{d-1} \r)$ satisfy 
\[
\# \{ \yy \cdot \xx : \xx \in E \} \ge N.
\]

\textbf{Observation 2.}  
Recall from Lemma \ref{Lem1.6} that for $d \ge 3$ we have  
\[
|X_\pi| = \begin{cases}
    q^{d-2} + q^{\frac{d-2}{2}} \, \eta\!\left( (-1)^{\frac{d}{2}} (\lambda^2-1)\right) , & \text{if $d$ is even and $\lambda \notin \{-1,1\}$},\\[6pt]
    q^{d-2} , & \text{if $d$ is even and $\lambda \in \{-1,1\}$},\\[6pt]
    q^{d-2} - q^{\frac{d-3}{2}} \big(\eta(-1)\big)^{\frac{d-1}{2}} , & \text{if $d$ is odd and $\lambda \notin \{-1,1\}$},\\[6pt]
    q^{d-2} + q^{\frac{d-1}{2}} \big(\eta(-1)\big)^{\frac{d-1}{2}} - q^{\frac{d-3}{2}} \big(\eta(-1)\big)^{\frac{d-1}{2}}, & \text{if $d$ is odd and $\lambda \in \{-1,1\}$}.
\end{cases}
\]
Thus, for $d \ge 3$ and sufficiently large $q$, we obtain  
\[
|X_\pi| \ge \frac{q^{d-2}}{2},
\]
except in the case $d=3$, $\lambda \in \{\pm 1\}$, and $q \equiv 1 \pmod 4$, where $|X_\pi| = 1$.  

Combining the above observations, it suffices to show that  
\[
\big| \Tca_1 \big( X_\pi ,E ,N \big) \big| < \frac{q^{d-2}}{3} \le \frac{2|X_\pi|}{3}.
\]

\textbf{Case 1:} $\boldsymbol{\lambda \notin \{\pm 1\}}$.  
By Theorem \ref{thm1.3}, for $d \ge 3$, $0 < N < q$, and $E \subset \Fbb_q^d$, we have  
\[
\big| \Tca_1 (X_\pi, E ,N) \big| \le \frac{Nq^{d-1}}{(q-N)|E|} + 
\begin{cases}
    \dfrac{Nq^{\frac{d-1}{2}}}{q-N}, & \text{if $d$ is odd},\\[6pt]
    \dfrac{Nq^{\frac{d-2}{2}}}{q-N}, & \text{if $d$ is even and $\eta \l( (-1)^{\frac{d}{2}} (\lambda^2-1)\r) = -1$},\\[6pt]
    \dfrac{N \big(q^{\frac{d}{2}} - q^{\frac{d-2}{2}}\big)}{q-N}, & \text{if $d$ is even and $\eta \l( (-1)^{\frac{d}{2}} (\lambda^2-1)\r) = 1$}.
\end{cases}
\]
For sufficiently large $q$, with $|E| \ge q$ and $N = \lfloor q/10 \rfloor$, we have $q - N \ge 0.9q$, hence  
\begin{align*}
\big| \Tca_1 (X_\pi ,E ,N) \big| 
&\le \frac{q^{d-1}}{9q} + 
\begin{cases}
    \dfrac{q^{\frac{d-1}{2}}}{9}, & \text{if $d$ is odd},\\[6pt]
    \dfrac{q^{\frac{d-2}{2}}}{9}, & \text{if $d$ is even, $\eta \l( (-1)^{\frac{d}{2}} (\lambda^2-1)\r) = -1$},\\[6pt]
    \dfrac{q^{\frac{d}{2}} - q^{\frac{d-2}{2}}}{9}, & \text{if $d$ is even, $\eta \l( (-1)^{\frac{d}{2}} (\lambda^2-1)\r) = 1$.}
\end{cases} \\
&\le \frac{q^{d-2}}{9} + \frac{q^{d-2}}{9} \le \frac{q^{d-2}}{3} \le \frac{2|X_\pi|}{3}.
\end{align*}

\textbf{Case 2:} $\boldsymbol{\lambda \in \{\pm 1\}}$.  
% \textbf{Case 2.1.} If $d\ge 4$. 
Applying Theorem \ref{p11}, for $d\ge 4, 0< N < q$, and $E \subset \Fbb_q^d$, we have  
\[
\big| \Tca_1 (X_\pi, E, N) \big| \le \frac{Nq^{d-1}}{(q-N)|E|} + 
\begin{cases}
    \dfrac{Nq^{\frac{d}{2}}}{q-N}, & \text{if $d$ is even},\\[6pt]
    \dfrac{Nq^{\frac{d-1}{2}}}{q-N}, & \text{if $d \equiv 3 \pmod{4}$ and $q \equiv 3 \pmod{4}$},\\[6pt]
    \dfrac{Nq^{\frac{d+1}{2}}}{q-N}, & \text{otherwise}.
\end{cases}
\]
For sufficiently large $q$, with $|E|\ge q$ and $N = \lfloor q/10 \rfloor$, we obtain  
\begin{align*}
\big| \Tca_1 (X_\pi , E,N) \big| 
&\le \frac{q^{d-1}}{9q} + 
\begin{cases}
    \dfrac{q^{\frac{d}{2}}}{9}, & \text{if $d$ is even},\\[6pt]
    \dfrac{q^{\frac{d-1}{2}}}{9}, & \text{if $d \equiv 3 \pmod{4},\, q \equiv 3 \pmod{4}$},\\[6pt]
    \dfrac{q^{\frac{d+1}{2}}}{9}, & \text{otherwise}
\end{cases}\\
&\le \frac{q^{d-2}}{9} + \frac{q^{d-2}}{9} \le \frac{q^{d-2}}{3} \le \frac{2|X_\pi|}{3}.
\end{align*}

% {\bf Case 2.2.} If $d=3$ and $q\equiv 1 \pmod 4$. ..................

This completes the proof.
\end{proof}

% {\color{blue} \begin{remark}
%     In the above proof, there remains the case when $d = 3$, $\lambda \in \{-1,1\}$, and $q \equiv 1 \pmod{4}$. 
    
%     If we apply Theorem \ref{p11} with $N = \frac{q}{a}$ for some $0 < a < 1$, we obtain
%     \[
%         \lv \Tca_1 \l( X_\pi ,E , \frac{q}{a} \r)\rv \le \frac{a q^{2}}{(1-a)|E|} + \frac{a q^2}{1-a}.
%     \]
%     When $q$ is sufficiently large, we have
%     \[
%         \frac{a q^2}{1-a} > 2q - 1 = |X_\pi|.
%     \]
%     Hence, Theorem \ref{p11} does not yield any useful information in this case.

%     If we use the trivial bound (Corollary \ref{cor_Eq2}), we have 
%     \[ \Tca_1 \l( X_\pi , E , N = \frac{q}{2} \r) \le \frac{Nq^d}{(q-N)|E|} =  \frac{q^3}{|E|}.  \]
%     Therefore, as similar arguments as the proof of Theorem \ref{thm1.7}, we can conclude that if $|E| \ge q^2$, then there are at least $q-1 = \frac{|X_\pi|-1}{2} $ elements $\yy$ in $\pi \cap S^{d-1}$ such that the set $\lo \yy \cdot \xx \colon \xx \in E \ro $ has size at least $\frac{q}{2}$. 
% \end{remark}}

\begin{remark}
    In the above proof, the remaining case is: $d = 3$, $\lambda \in \{-1,1\}$, and $q \equiv 1 \pmod{4}$. 
    
    If we apply Theorem \ref{p11} with $N = \frac{q}{C}$ for some $0 < C < 1$, we obtain
    \[
        \lv \Tca_1 \l( X_\pi ,E , \frac{q}{C} \r)\rv \le \frac{C q^{2}}{(1-C)|E|} + \frac{C q^2}{1-C}.
    \]
    For sufficiently large $q$, it follows that
    \[
        \frac{C q^2}{1-C} > 2q - 1 = |X_\pi|,
    \]
    and thus Theorem \ref{p11} fails to provide any meaningful information in this case.
\end{remark}

\section{Bounding $\T_1(X, E, N)$ for $\lambda=0$ (Theorem \ref{thm1.55})}
Theorem \ref{thm1.55} follows from the following general statement.
\begin{theorem}\label{thm1.55'}
    Let $E\subset \mathbb F_q^d$ and $\lambda=0$.
    Then we have
 $$ \lv \mathcal{T}_1\l(X_\pi, E, N\r) \rv \ll \begin{cases}
    \dfrac{q^{\frac{3}{2}}\lv S^{d-2}\rv N}{(q-N)|E|^{\frac{1}{2}}} , & \text{if $d=3$},\\[10pt]
    \dfrac{q^{\frac{2d-2}{d}}\lv S^{d-2} \rv N}{(q-N)|E|^{\frac{d-2}{d}}}, & \text{for other cases}.
\end{cases} $$
 For $\mathbf{x}\in \mathbb{F}_q^{d-1}$, define $
f(\mathbf{x}) := \#\{t \in \mathbb{F}_q : (\mathbf{x},t) \in E\}
$, then 
$$ \lv \mathcal{T}_1\l(X_\pi, E, N\r) \rv \le \frac{2q^{d-1}N}{(q-N)|E|} \cdot \max_{\xx\in \mathbb{F}_q^{d-1}}f(\xx).$$
\end{theorem}

The three inequalities in Theorem~\ref{thm1.55'} play slightly different roles.

If
$N = q^\alpha$ with $0<\alpha<1$ and $|E| = q^\beta$, then one checks that the first two bounds, obtained from Fourier extension estimates for the
sphere $S^{d-2}\subset \mathbb{F}_q^{d-1}$, are non-trivial precisely in the regimes
\[
\beta > 1 + 2\alpha \quad\text{when } d=3,
\qquad
\beta > 1 + \frac{d}{d-2}\,\alpha \quad\text{when } d\ge 4.
\]
We do not know whether these exponents are sharp. By contrast, the third inequality
\[
\bigl|\mathcal{T}_1(X_\pi,E,N)\bigr|
\;\le\;
\frac{2q^{d-1}N}{(q-N)|E|}\cdot\max_{\mathbf{x}\in\mathbb{F}_q^{d-1}} f(\mathbf{x}),
\]
has the correct order of magnitude in terms of its dependence on $N$ and
$\max f$. Construction~\ref{cons1.5} shows that this estimate is
essentially optimal, and its structure is very close to the bounds obtained in
Theorems~\ref{thm1.3} and~\ref{p11}. In this sense, the last inequality
captures the \textit{right scale} for exceptional sets in the $\lambda=0$ case,
while the first two inequalities highlight the additional information that
can be extracted from Fourier extension estimates.

To prove this theorem, we make use of the following analog of Lemma \ref{Lem1.9}.

\begin{lemma} \label{Lem1.912}Let $E\subset \mathbb F_q^d.$ Then the following statements hold.
\begin{itemize} 
\item [(i)] If $d=3$, then
$$  \sum_{V\in X_\pi} \sum_{\xi \in V\setminus \{\oo\}} |\widehat{E}(\xi)|^2 \ll q^{3/2}\lv S^{d-2} \rv |E|^{3/2}.$$

\item [(ii)] For other cases, then
$$   \sum_{V\in X_\pi} \sum_{\xi \in V\setminus \{\oo\}} |\widehat{E}(\xi)|^2 \ll q^{\frac{2d-2}{d}} \lv S^{d-2}\rv |E|^{\frac{d+2}{d}}.$$
\end{itemize}
\end{lemma}

\begin{proof}
For $\alpha\in E$, define $f(\alpha')$ to be the number of elements $\alpha$ in $E$ with $(\alpha_1, \ldots, \alpha_{d-1})=\alpha'$. 
\begin{align*}  \sum_{V\in X_\pi} \sum_{\xi \in V\setminus \{\oo\}} |\widehat{E}(\xi)|^2 =& \sum_{k \in \Fbb_q^\ast} \sum_{\substack{\xi \in \Fbb_q^d \colon\\  \xi_1^2 +\cdots + \xi_{d-1}^2=1,\\ \xi_d =0}}  \lv \widehat{E}(k\xi)\rv^2 
= \sum_{k\in \Fbb_q^\ast} \sum_{\substack{\xi\in \mathbb F_q^d:\\ \xi_1^2+\cdots+ \xi_{d-1}^2=1,\\ \xi_d=0}}  \sum_{\alpha, \beta\in E} \chi( (\alpha-\beta)\cdot k\xi) \\
=&\sum_{k\in \Fbb_q^\ast}\sum_{\xi\in \mathbb{F}_q^{d-1}, ||\xi||=1}\sum_{\alpha', \beta'\in \mathbb{F}_q^{d-1}}f(\alpha')f(\beta')\chi(\alpha'-\beta')\cdot k\xi)\\
=&\sum_{k \in \Fbb_q^\ast}\sum_{\xi\in S^{d-2}}|\widehat{f}(k\xi)|^2. 
%\le  \sum_{\xi \in \Fbb_q^{d-1}} \lv \widehat{f}(\xi)\rv^2   = q^{d-1}\sum_{\xx \in \Fbb_q^{d-1}} \lv f(\xx)\rv^2.
\end{align*}
To bound $\sum_{k \in \Fbb_q^\ast}\sum_{\xi \in S^{d-2}} \lv \widehat{f}(k\xi) \rv^2$, we apply Theorem \ref{lem6.1} for $\widehat{f}$. More precise, for each $k \in \Fbb_q^\ast$, it is clear that $kS^{d-2}= S^{d-2}_{k^2}$, so, one has 
\begin{align*}
    \l( \sum_{\xx\in \Fbb_q^{d-1} }\lv \frac{1}{\lv S^{d-2}\rv} \sum_{\xi \in S^{d-2}} \chi (-k\xi \cdot \xx) \widehat{f}(k\xi) \rv^{\frac{2d}{d-2}}  \r)^{\frac{d-2}{2d}} \ll \l( \frac{1}{\lv S^{d-2} \rv} \sum_{\xi\in S^{d-2}} \lv \widehat{f}(k\xi) \rv^2 \r)^{\frac{1}{2}}. 
\end{align*}
From this, by H\"{o}lder inequality, for each $k \in \Fbb_q^\ast$ we get 
\begin{align*}
    \sum_{\xi \in S^{d-2}} \lv \widehat{f}(k\xi)\rv^2 & = \sum_{\xi \in S^{d-2}} \widehat{f}(k\xi)\overline{\widehat{f}(k\xi)}  
      = \sum_{\xx \in \Fbb_q^{d-1}} f(\xx) \overline{\sum_{\xi \in S^{d-2}} \chi (k\xi \cdot \xx) \widehat{f}(k\xi)} \\
    & \le \l( \sum_{\xx \in \Fbb_q^{d-1}} \lv f(\xx)\rv^{\frac{2d}{d+2}} \r)^{\frac{d+2}{2d}} \l( \sum_{\xx \in \Fbb_q^{d-1}} \lv \sum_{\xi \in S^{d-2}} \chi (k\xi \cdot \xx ) \widehat{f}(k\xi) \rv^{\frac{2d}{d-2}} \r)^{\frac{d-2}{2d}} \\
    & = \l( \sum_{\xx \in \Fbb_q^{d-1}} \lv f(\bfx)\rv^{\frac{2d}{d+2}} \r)^{\frac{d+2}{2d}} \l( \sum_{\xx \in \Fbb_q^{d-1}} \lv \sum_{\xi \in S^{d-2}} \chi (-k\xi \cdot \xx ) \widehat{f}(k\xi) \rv^{\frac{2d}{d-2}} \r)^{\frac{d-2}{2d}} \\
    & \ll  \l( \sum_{\xx \in \Fbb_q^{d-1}} \lv f(\bfx)\rv^{\frac{2d}{d+2}} \r)^{\frac{d+2}{2d}} \l( \lv S^{d-2}\rv \sum_{\xi \in S^{d-2}} \lv \widehat{f}(k \xi)\rv^2  \r)^{\frac{1}{2}}. 
\end{align*}
 
Thus,
\[\sum_{\xi\in S^{d-2}}|\widehat{f}(k\xi)|^2\ll \lv S^{d-2}\rv \left(\sum_{x\in \mathbb{F}_q^{d-1}}|f(x)|^{\frac{2d}{d+2}}\right)^{\frac{d+2}{d}},\]
for each $k \in \Fbb_q^\ast$. Note that $f(\xx)\le q$ for all $\xx\in \mathbb{F}_q^{d-1}$. Therefore, 
\[\sum_{\xx\in \mathbb{F}_q^{d-1}}|f(\xx)|^{\frac{2d}{d+2}}\le q^{\frac{d-2}{d+2}}\cdot\sum_{\mathbf{x}\in \mathbb{F}_q^{d-1}}f(\mathbf{x})=q^{\frac{d-2}{d+2}}\cdot |E|.\]
This implies that 
\[ \sum_{V\in X_\pi} \sum_{\xi \in V\setminus \{\oo\}} |\widehat{E}(\xi)|^2 \ll (q-1)\lv S^{d-2}\rv \left(q^{\frac{d-2}{d+2}}|E|\right)^{\frac{d+2}{d}}\le \lv S^{d-2} \rv q^{\frac{2d-2}{d}}|E|^{\frac{d+2}{d}}.\]

When $d=3$, by the same argument but using Theorem \ref{lem6.2} instead of Theorem \ref{lem6.1}, one has a better estimate, namely, 
\[\sum_{\xi\in S^{1}}|\widehat{f}(k\xi)|^2 \ll \lv S^{d-2}\rv \left(\sum_{\xx\in \mathbb{F}_q^{2}}|f(\xx)|^{\frac{4}{3}}\right)^{\frac{3}{2}},\]
for each $k \in \Fbb_q^\ast$.
Hence, 
\[  \sum_{V\in X_\pi} \sum_{\xi \in V\setminus \{\oo\}} |\widehat{E}(\xi)|^2 \ll \lv S^{d-2}\rv (q-1)q^{\frac{1}{2}}|E|^{\frac{3}{2}} \le  \lv S^{d-2}\rv q^{\frac{3}{2}}|E|^{\frac{3}{2}}.\]
This completes the proof.
\end{proof}
We now prove Theorem \ref{thm1.55'}.
\begin{proof}[Proof of Theorem \ref{thm1.55'}]
 By Theorem \ref{lem1K} and recall that $\mathcal{T}_1 (X_\pi , E ,N) \subset X_\pi = [\pi \cap S^{d-1}]$, it follows that
\begin{align*}
    \lv \mathcal{T}_1\l(X_\pi, E, N\r)\rv & \le  \frac{N}{(q-N)|E|^2}  \sum_{V\in \T_1(X_\pi, E,N)} \sum_{\xi \in V\setminus \{\oo\}} |\widehat{E}(\xi)|^2 \\
    & \le \frac{N}{(q-N)|E|^2}  \sum_{V\in X_\pi} \sum_{\xi \in V\setminus \{\oo\}} |\widehat{E}(\xi)|^2.
\end{align*}
Now, applying Lemma \ref{Lem1.912} we get the first estimate. 

To prove the second estimate, we proceed in an alternative way
\begin{align*}
    \lv \mathcal{T}_1\l( X_\pi, E, N\r) \rv & \le \frac{N}{(q-N)|E|^2}  \sum_{V\in X_\pi} \sum_{\xi \in V\setminus \{\oo\}} |\widehat{E}(\xi)|^2 
    = \frac{N}{(q-N)|E|^2}  \sum_{k\in \Fbb_q^\ast} \sum_{\xi \in \pi \cap S^{d-1} } |\widehat{E}(k\xi)|^2 \\
    & = \frac{N}{(q-N)|E|^2}  \sum_{k\in \Fbb_q^\ast} \sum_{\substack{\xi \in \Fbb_q^d \colon \\ \xi_1^2 +\cdots +\xi_{d-1}^2 =1 , \xi_d =0}} |\widehat{E}(k\xi)|^2\\
&=  \frac{N}{(q-N)|E|^2} \sum_{k\in \Fbb_q^\ast}\sum_{\xi\in \mathbb{F}_q^{d-1}, ||\xi||=1}\sum_{\alpha', \beta'\in \mathbb{F}_q^{d-1}}f(\alpha')f(\beta')\chi(\alpha'-\beta')\cdot k\xi)\\
&= \frac{N}{(q-N)|E|^2} \sum_{k \in \Fbb_q^\ast}\sum_{\xi\in S^{d-2}}|\widehat{f}(k\xi)|^2 
\le \frac{N}{(q-N)|E|^2} \cdot 2\sum_{\xi \in \Fbb_q^{d-1}} \lv \widehat{f}(\xi)\rv^2. 
\end{align*} 
Thus, it is enough to bound the following
$$ \lv \mathcal{T}_1\l( X_\pi, E, N\r) \rv\le \frac{2N}{(q-N)|E|^2}  \sum_{\xi \in \mathbb{F}_q^{d-1}}|\widehat{f}(\xi)|^2,$$

where $f(\alpha')$ is defined as the number of elements $\alpha$ in $E$ with $(\alpha_1, \ldots, \alpha_{d-1})=\alpha'$. 

Therefore, by Plancherel identity and note that $\sum_{\xx \in \Fbb_q^{d-1}} f(\xx ) = |E|$, one has
\begin{align*}
    \lv \mathcal{T}_1\l( X_\pi, E, N\r) \rv & \le  \frac{2N}{(q-N)|E|^2}  \sum_{\xi \in \mathbb{F}_q^{d-1}}|\widehat{f}(\xi)|^2 \\
    & = \frac{2N}{(q-N)|E|^2} \cdot  q^{d-1} \sum_{\xx \in \Fbb_q^{d-1}} \lv f(\xx) \rv^2 \\
    & \le \frac{2N}{(q-N)|E|^2} \cdot  q^{d-1} \cdot \l( \max_{\xx \in \Fbb_q^{d-1}} f(\xx) \r) \cdot  \sum_{\xx \in \Fbb_q^{d-1}}  f(\xx)   \\
    &  = \frac{2N}{(q-N)|E|}\cdot q^{d-1}\cdot \max_{\xx\in \mathbb{F}_q^{d-1}}f(\xx).
\end{align*} 
This completes the proof.
\end{proof}

% \begin{remark}
%     For any $d\ge 3$, $E \subset \Fbb_q^d$ and, we have
%     \begin{align*}
%         \lv \Tca_1 ([S^{d-2}], E , N) \rv = \lv \Tca_1 (X_\pi , E\times \Fbb_q ,N )\rv .
%     \end{align*}
%     So, from Proposition \ref{p0} and \cite[Theorem 7.1]{BHIPR17}, one has
%     \begin{align*}
%         \lv \Tca_1 ([S^{d-2}], E ,N)\rv &\ll \begin{cases}
%             \frac{q^{\frac{3}{2}} \lv S^{d-2}\rv N}{(q-N)q^{\frac{1}{2}} |E|^{\frac{1}{2}}} , & \text{if $d=3$},\\
%         \frac{q^{\frac{2d-2}{d}} \lv S^{d-2} \rv N}{(q-N)q^{\frac{d-2}{d}}|E|^{\frac{d-2}{d}}}, & \text{for other cases},
%         \end{cases}\\
%         & \sim \begin{cases}
%             \frac{q^{d-1} N}{(q-N)|E|^{\frac{1}{2}}} , & \text{if $d=3$},\\
%         \frac{q^{d-1} N}{(q-N)|E|^{\frac{d-2}{d}}}, & \text{for other cases},
%         \end{cases}
%     \end{align*}
% \end{remark}

% \begin{example}
%     For $d\ge 3$, let $E' \subset \Fbb_q^{d-1}$. Let $E:= \bigcup_{u \in E'} \ell_u' $, where
%     \[ \ell_u' : = \lo (u,k) \colon k\in \Fbb_q \ro \subset \Fbb_q^d. \]
%     Then, for $X_\pi$ with $\lambda =0$ and $0 < N \le q$, we have 
%     \begin{align*}
%         \lv \Tca_1 (X_\pi ,E , N)\rv = \lv \Tca_1 ([S^{d-2}], E ' ,N ) \rv  
%     \end{align*}
% \end{example}

\section{Bounding $\T_2(X, E,N )$ (Theorem \ref{thm:1.5})}
% \thang{Hung, please change $n$ to $\mathbf{n}$ in the following paragraph.}

While restriction type lemmas (Lemmas \ref{Lem1.9}, \ref{Lem1.911}, and \ref{Lem1.912}) play a crucial role in bounding the size of the exceptional set $\mathcal{T}_1(X_\pi,E,N)$, the following geometric result is sufficient in estimating the size of $\mathcal{T}_2(X_\pi,E,N)$.
\begin{lemma}\label{lm2.3}
    Let $\ell$ be a line passing through the origin in $\mathbb{F}_q^d$. 
    \begin{enumerate}
        \item If $q\equiv 3\pmod 4$, then $|\ell^\perp\cap \{x_d=\lambda\}\cap S^{d-1}|\le \begin{cases}
           2q^{d-2}, & \text{if $\lambda =0$},\\
           2q^{d-3}, & \text{otherwise}.
        \end{cases}$ 
        \item If $q\equiv 1\pmod 4$, then 
        $|\ell^\perp\cap \{x_d=\lambda\}\cap S^{d-1}|\le \begin{cases}
           2q^{d-2}, & \text{if $\lambda =0$},\\
           3q^{d-3}, & \text{otherwise}.
        \end{cases}$
    \end{enumerate}
\end{lemma}
\begin{proof}
     Let $\yy= (y_1,y_2,\ldots, y_d)$ be the direction vector of $\ell$. Then, we have
        \[ \ell^\perp = \lo \xx =(x_1,x_2 ,\dots ,x_d) \in \F_q^d\colon \sum_{i=1}^d x_iy_i =0 \ro ,\]
        so that 
        \[ \ell^\perp \cap \lo x_d =\lambda \ro = \lo \xx\in \F_q^d \colon x_d=\lambda , \sum_{i=1}^{d-1}x_iy_i = -\lambda y_d \ro .\]
        On the other hand, we have 
        \[ \lo x_d =\lambda \ro \cap S^{d-1} = \lo \xx \in \F_q^d \colon \sum_{i=1}^{d-1} x_i^2 = 1 -\lambda^2 \ro . \]
        Therefore, it is sufficient to count all solutions $\l( x_1 ,x_2\ldots ,x_{d-1} \r) \in \F_q^{d-1}$ of
        \[ \begin{cases}
            \sum_{i=1}^{d-1}x_iy_i = -\lambda y_d  , \\
            \sum_{i=1}^{d-1} x_i^2 = 1 -\lambda^2  .
        \end{cases} \]
 We now consider three cases as follows.

        {\bf Case 1}: There exists $1\le i_0 \le d-1$ such that $y_{i_0} \ne 0$. With $x_{i_0} = y_{i_0}^{-1}\l( -\lambda y_d -\sum_{i\ne i_0} x_iy_i \r) $, we have
        \begin{align}\label{eq:10}
            \sum_{i\ne i_0} x_i^2 + y_{i_0}^{-2}\l( -\lambda y_d - \sum_{i\ne i_0} x_iy_i \r)^2 =1-\lambda^2. 
        \end{align} 
  When $q \equiv 3 \pmod 4$, there exists $i_1\ne i_0$ such that $1+y_{i_0}^{-2}y_{i_1}^2 \ne 0$. Therefore, for each choice $x_i \in  \F_q$ with $i \not\in\{ i_1,i_0\}$, the above equation becomes a quadratic equation in the variable $x_{i_1}$. Therefore, it has at most two solutions $x_{i_1}$. Therefore, the above equation has at most $2q^{d-3}$ solutions. So,
        \[ |\ell^\perp\cap \{x_d=\lambda\}\cap S^{d-1}|\le 2q^{d-3}. \]
 When $q \equiv 1 \pmod 4$, if there exists $i_1\ne i_0$ such that $1+y_{i_0}^{-2}y_{i_1}^{2} \ne 0$, 
            A similar argument implies
            \[ |\ell^\perp\cap \{x_d=\lambda\}\cap S^{d-1}|\le 2q^{d-3}. \]
            Otherwise, we have $y_{i_0}^{-2}y_{i}^2 =-1$ for all $i \ne i_0$. Fix $i_1 \ne i_0$, for each choice $x_i \in \F_q$ with $i \not\in\{i_0,i_1\}$ such that
            \[ -\lambda y_d - \sum_{i \ne i_1 ,i_0} x_iy_i \ne 0, \]
            the equation (\ref{eq:10}) becomes a linear equation in $x_{i_1}$. Thus, it has at most one solution. 
            
            For each choice $x_i \in \F_q$ with $i \not\in \{i_0,i_1\}$ such that 
            \[ -\lambda y_d - \sum_{i \ne i_1 ,i_0} x_iy_i =0, \]
            the equation (\ref{eq:10}) has at most $q$ solutions $x_{i_1}\in \mathbb{F}_q$. Next, we count the number of solutions of the equation
            \begin{align}\label{eq:11}
                -\lambda y_d - \sum_{i \ne i_1 ,i_0} x_iy_i =0 . 
            \end{align} 
            Since $y_0^{-2}y_i^2 = -1$ for all $i \not\in \{i_0,i_1\}$, we have $y_i \ne 0$ for all $i \not\in\{ i_0,i_1\}$. Fix $i_2\not\in \{i_0,i_1\}$, for each choice of $x_i \in \F_q$ with $i \not\in \{i_0,i_1,i_2\}$, the equation (\ref{eq:11}) becomes a linear equation in variable $x_{i_2}$. Hence, it has at most one solution. This implies at most $q^{d-4}$ solutions. Therefore, the number of solutions of the equation (\ref{eq:10}) is at most
            \[ 2(q^{d-3}-q^{d-4}) + q^{d-4} \cdot q \le 3 q^{d-3}. \]

        {\bf Case 2:} If $y_1=\cdots =y_{d-1}=0$, we have 
        \[\ell^\perp \cap \lo x_d=\lambda \ro =\begin{cases}
            \lo x_d =\lambda \ro , & \text{if $\lambda =0$},\\
            \emptyset , & \text{otherwise}.
        \end{cases}  \]

         {\bf Case 2.1:} If $\lambda =0$, then we consider the following equation
         \[ \sum_{i=1}^{d-1} x_i^2 =1. \]
         For each choice $x_i \in  \F_q$ with $i \ne 1$, there exist at most two possible solutions of $x_1$, so we have at most  $2q^{d-2}$ solutions at the end. Then,
         \[ |\ell^\perp\cap \{x_d=\lambda\}\cap S^{d-1}|\le 2q^{d-2}. \]

         {\bf Case 2.2:} If $\lambda \ne 0$, then from $\ell^\perp \cap \lo  x_d = \lambda \ro = \emptyset$, one has
         \[ |\ell^\perp\cap \{x_d=\lambda\}\cap S^{d-1}| =0 .\]
This completes the proof.
\end{proof}

\begin{proof}[Proof of Theorem \ref{thm:1.5}]
    For $V\in X_{\pi}$,
    % , by $\ell_e$ we mean the line passing through the origin and $e$. 
 let $\{ \xx_{V_j}+V\}_{j=1}^{q^{d-1}}$ be the set of disjoint translations of $V \subset \mathbb F_q^d.$
We have 
\[|E|=\sum_{j=1}^{q^{d-1}}|E\cap (\xx_{V_j}+ V )|.\]
Using the Cauchy-Schwarz inequality, we have 
\[|E|^2\le |\pi_{V^\perp}(E)|\sum_{j=1}^{q^{d-1}}|E\cap (\xx_{V_j}+V)|^2.\]
Thus, 
\[|E|^2|\T_2 (X_\pi ,E,N)|\le \sum_{V\in \mathcal{T}_2(X, E, N)}|\pi_{V^\perp}(E)|\sum_{j=1}^{q^{d-1}}|E\cap (\xx_{V_j}+V)|^2.\]
This gives
\[|E|^2|\T_2 (X_\pi ,E, N)|\le N \sum_{V\in \mathcal{T}_2(X, E, N)}\sum_{j=1}^{q^{d-1}}|E\cap (\xx_{V_j}+V)|^2.\]
So, by Lemma \ref{Lem3.1KP}, 
\[|E|^2|\T_2 (X_\pi ,E,N)|\le Nq^{-d+1} \sum_{V \in \mathcal{T}_2(X, E, N)}\sum_{\xi\in V^\perp}|\widehat{E}(\xi)|^2.\]
The right hand side at $\xi =\oo$ is at most $Nq^{-d+1}|E|^2|\T_2 (X_\pi ,E ,N)|$, which is at most half of the left hand side if $N< \frac{q^{d-1}}{2} $.

Define $M=\max_{\xi \ne 0}\#\{ V\in \T_2 (X_\pi ,E, N) \colon \xi \in V^\perp\}$.

For $\xi \ne \oo$, let $\ell_\xi$ be the line passing through the origin and $\xi$. It is not hard to see that $\xi \in V^\perp$ if and only if $V\subset \ell_\xi ^\perp$. Indeed, since $k\xi \cdot \xx = k \cdot 0 =0, \, \forall k\in \mathbb{F}_q$ and $\xx\in V$, we have $\xx \in \ell_\xi^\perp ,\, \forall \xx \in V$ and so $V \subset \ell_\xi^\perp$. 

By Lemma \ref{lm2.3}, we know that 
\[ |\ell_\xi^\perp\cap X_{\pi}|\le \begin{cases}
   2q^{d-2}, & \text{if $\lambda =0$}\\
   3q^{d-3}, & \text{otherwise}.
\end{cases}  .\] 
Thus, 
\[ M\le \begin{cases}
    2q^{d-2}, & \text{if $\lambda =0$},\\
    3q^{d-3}, & \text{otherwise}.
\end{cases}. \]

%Note that $\ell_m^\perp\cap \{x_d=\lambda\}$ be a $d-2$-plane, and a $d-2$-plane intersects the sphere $S^2$ in at most $2q^{d-3}$ points under the assumption that $q\not\equiv 1\pmod 4$ or $\lambda \not\in \{ 1,-1 \}$. Thus, $M\le 2q^{d-3}$. 

% The above inequality can now be rewritten as 
% \[\frac{|E|^2|\T_2 (X_\pi ,E,N )|}{2}\le 3Nq^{-d+1} q^{d-3} \sum_{m\in \mathbb{F}_q^d}|\widehat{E}(m)|^2\le 2Nq^{d-2}|E|.\]
% Hence, $|\T_2 (X_\pi ,E,N )|\le \frac{4Nq^{d-2}}{|E|}$. This completes the proof of the theorem.

The above inequality can now be rewritten as 
\begin{align*}
    \frac{|E|^2|\T_2 (X_\pi ,E,N )|}{2}& \le \begin{cases}
    2Nq^{-d+1}q^{d-2} \sum_{\xi \in \Fbb_q^d} \lv \widehat{E}(\xi )\rv^2, & \text{if $ \lambda =0$}, \\
    3Nq^{-d+1} q^{d-3} \sum_{\xi \in \mathbb{F}_q^d}|\widehat{E}(\xi)|^2, & \text{otherwise},
\end{cases} \\
& \le \begin{cases}
    2Nq^{d-1}|E|, & \text{if $\lambda =0$},\\
    3Nq^{d-2}|E|, & \text{otherwise}.
\end{cases}
\end{align*}
Hence, 
 
\[ |\T_2 (X_\pi ,E,N )|\le \begin{cases}
    \dfrac{4Nq^{d-1}}{|E|}, & \text{if $\lambda =0$},\\[10pt]
    \dfrac{6Nq^{d-2}}{|E|}, & \text{otherwise}.
\end{cases}  \] 
This completes the proof of the theorem.

\end{proof}

\section{Different radii -- similar results}
Let $S^{d-1}_{-1}$ and $S^{d-1}_0$ be the spheres centered at the origin of radius $-1$ and $0$ in $\Fbb_q^d$, respectively, i.e,
\[ S_{-1}^{d-1}:= \lo \bfx \in \Fbb_q^d \colon \lV \bfx \rV =-1 \ro , \quad S_0^{d-1}:= \lo \bfx  \colon \lV \bfx \rV =0 \ro . \]
Define $Y_\pi := \lu\pi \cap S^{d-1}_{-1}\ru$ and $Z_\pi = \lu \pi\cap S_0^{d-1}  \ru $.

The same approach implies similar results for $Y_\pi$ and $Z_\pi$, so we omit the proofs.
\begin{theorem}\label{thm1.4}
    Let  $E\subset \mathbb F_q^d$ and $\lambda \in \Fbb_q \setminus \lo \pm i \ro$. 
     \begin{itemize}
    % \item [(a)]  If $d=3$, then  $\lv \mathcal{T}_1\l(Y_\pi, E, N \r)\rv \le \frac{Nq^d}{(q-N)|E|}.$
    \item [(a)] If $d\ge 3$ is odd then 
 \[ \lv \mathcal{T}_1\l( Y_\pi, E,N\r) \rv \le \frac{Nq^{d-1}}{(q-N)|E|} + \frac{Nq^{\frac{d-1}{2}}}{q-N}.\]
\item [(b)] If $d\ge 4$ is even and $\eta\left( (-1))^{\frac{d}{2}} (\lambda^2+1) \right)=-1,$ then
\[\lv \mathcal{T}_1\l( Y_\pi, E,N\r) \rv \le \frac{Nq^{d-1}}{(q-N)|E|} + \frac{Nq^{\frac{d-2}{2}}}{q-N} .\]

\item [(c)] If $d\ge 4$ is even and $\eta\left( (-1))^{\frac{d}{2}} (\lambda^2+1) \right)=1,$ then
    \[ \lv \mathcal{T}_1\l(Y_\pi, E,N\r) \rv\le \frac{Nq^{d-1}}{(q-N)|E|} + \frac{N \l(q^{\frac{d}{2}} -q^{\frac{d-2}{2}}\r)}{q-N} .\]

\end{itemize}
\end{theorem}

\begin{theorem}\label{p11i}  Let $E\subset \mathbb F_q^d$, and $\lambda \in \{\pm i\}$. We have
\begin{align*}
    \lv \mathcal{T}_1\l( Y_\pi, E, N\r) \rv \le \begin{cases}
        \dfrac{q^{d-1}N}{(q-N)|E|}  + \dfrac{q^{\frac{d}{2}}N}{q-N} ,& \text{if $d \ge 4$ and even},\\[10pt]
        \dfrac{q^{d-1}N}{(q-N)|E|}  + \dfrac{q^{\frac{d-1}{2}}N}{q-N} ,&  \text{if $d \equiv 3 \pmod 4, q\equiv 3 \pmod 4$,}\\[10pt]
        \dfrac{q^{d-1}N}{(q-N)|E|}  + \dfrac{q^{\frac{d+1}{2}}N}{q-N} , & \text{otherwise}.
    \end{cases}
\end{align*}
\end{theorem} 

\begin{theorem}\label{pi02}  Let $E\subset \mathbb F_q^d$ and $\lambda=0$. We have
$$ \lv \mathcal{T}_1\l(Y_\pi, E, N\r) \rv \ll \begin{cases}
    \dfrac{q^{\frac{2d-1}{2}} N}{(q-N)|E|^{\frac{1}{2}}} , & \text{if $d=3$},\\[10pt]
    \dfrac{q^{\frac{d^2-2}{d}} N}{(q-N)|E|^{\frac{d-2}{d}}}, & \text{for other cases}.
\end{cases} $$
\end{theorem}

\begin{theorem}\label{p10}  Let $E\subset \mathbb F_q^d$. We have
\begin{align*}
    \lv \mathcal{T}_1\l( \lu S_0^{d-1}\ru, E, N\r) \rv \le \begin{cases}
        \dfrac{q^{d-1}N}{(q-N)|E|}  + \dfrac{q^{\frac{d-1}{2}}N}{q-N} ,& \text{if $d \ge 3$  is  odd},\\[10pt]
        \dfrac{q^{d-1}N}{(q-N)|E|}  + \dfrac{q^{\frac{d-2}{2}}N}{q-N} ,&  \text{if $d \equiv 2 \pmod 4, q \equiv 3 \pmod 4$},\\[10pt]
        \dfrac{q^{d-1}N}{(q-N)|E|}  + \dfrac{q^{\frac{d}{2}}N}{q-N} , & \text{otherwise}.
    \end{cases}
\end{align*}
\end{theorem} 

\begin{theorem}\label{thm:1.5''} Let $E\subset \mathbb{F}_q^d$. If $0< N< q^{d-1}$, then for $X \in \lo X_\pi ,Y_\pi .Z_\pi \ro$, we have
 \[ |\T_2 (X ,E,N )|\le \begin{cases}
    \frac{2Nq^{2d-2}}{\l( q^{d-1} -N \r)|E|}, & \text{if $\lambda =0$},\\
    \frac{3Nq^{2d-3}}{\l( q^{d-1}-N \r)|E|}, & \text{otherwise}.
\end{cases}   \]   
\end{theorem}
\begin{remark}
The change from the conditions $\lambda \notin \{0,\pm 1\}$ in Theorems \ref{thm1.3}--\ref{p11}
to $\lambda \notin \{\pm i\}$ in Theorem~\ref{thm1.4} reflects the change of radius:
for slices of the sphere $\|x\| = 1$, the relevant quadratic parameter is
$\lambda^2 - 1$, whereas for slices of $\|x\| = -1$, the parameter becomes
$\lambda^2 + 1$, so the degenerate cases are $\lambda^2 + 1 = 0$
(i.e.\ $\lambda = \pm i$).
\end{remark}

\section{Sharpness examples}
We need to keep in mind an observation that if 
\[ \lv \Tca_1 \l( X_\pi , E , N \r) \rv< \frac{|X_\pi|}{C},\]
for some positive $C>1$, then there will be at least $\left(1-\frac{1}{C}\right)|X_{\pi}|$  points $\yy \in \l( \pi \cap S^{d-1}\r) $ such that 
% the size of  
$  \lv \lo \mathbf{y}\cdot\mathbf{x} \colon \mathbf{x}\in E \ro \rv$  is at least $N$. 
% \thang{Change the notation $\mathbf{y}$ accordingly}
\begin{construction}\label{fcon}
    
For all $\lambda\in \mathbb{F}_q$, there exists $E\subset \mathbb{F}_q^d$ with $|E|=q$ such that for all $\mathbf{y}\in \mathbb{F}_q^{d-1}\times \{\lambda\}$ one has $|\mathbf{y}\cdot E|\le \frac{q+1}{2}$. Indeed, if $\lambda=0$, we can set $E=\{0\}^{d-1}\times \mathbb{F}_q$. If $\lambda \ne 0$, we can set $E=\{0\}^{d-2}\times \{(t, t^2)\colon t\in \mathbb{F}_q\}$. Using the fact that the number of square elements is $\frac{q+1}{2}$, we derive $|\mathbf{y}\cdot E|\le \frac{q+1}{2}$.
This implies that
\[ \Tca_1 \l( X_\pi , E , \frac{q+1}{2} \r) =X_\pi.\]
\end{construction}

\begin{construction}\label{cons1.5}
    Let $E= E' \times \Fbb_q \subset \Fbb_q^d$, where $ E ' = \lo \xx_1 ,\xx_2 ,\ldots ,\xx_N\ro\subset \mathbb{F}_q^{d-1}$ with $0< N < \frac{q}{2}$. 
    For $\lambda =0$, for all $V\in X_\pi$ we have $ (0, \cdots ,0,1) \in V^\perp$, so $A:=\lo (0,\cdots ,0 ,t) \colon t \in \Fbb_q \ro \subset V^\perp$. Therefore,
     \[  \l( \{ \xx_i\} \times \Fbb_q \r) = \l( \l( \xx_i ,0 \r) + A \r)  \subset  \l( \l(\xx_i,0 \r)+ V^\perp \r) , \quad \forall \, i=1,2,\ldots ,N. \]
This implies that $E$ can be covered by at most $N$ translates of $V^\perp$. 
    Then, $\lv \pi_V (E) \rv \le N$, so
    \[ \lv  \Tca_1 \l(  X_\pi ,E ,N \r) \rv = \lv X_\pi \rv \sim q^{d-2}. \]
    On the other hand, $\max_{\xx\in \Fbb_q^{d-1}} f(\xx) = f(\xx_1) = f(\xx_2) = \cdots = f(\xx_{N})=q$, so from Theorem \ref{thm1.55}, one has
    \[  \lv X_\pi \rv = \lv \Tca_1 \l( X_\pi ,E,N \r) \rv \le \frac{2q^{d-1}N}{(q-N)|E|} \cdot \max_{\xx\in \Fbb_q^{d-1}} f(\xx) = \frac{2q^{d-1}N}{(q-N)qN}\cdot q \le 4q^{d-2} \sim \lv X_\pi \rv .\]
\end{construction}
\begin{construction}\label{16c}
 Let $E = \Fbb_q^{d-1}\times I $, where $I\subset \mathbb{F}_q$, $ |I| =\left \lfloor \frac{|E|}{q^{d-1}} \right \rfloor $,   and $0 \in I$.
    For $\lambda \ne 0$, we have $V \notin (\mathbb{F}_q^{d-1}\times \{0\})=(0,\cdots ,0,1)^\perp$ for all $V \in X_\pi $ . 
     Therefore,  $ V \cap \l( \Fbb_q^{d-1} \times \{ 0 \} \r)  =\{ \oo \}$. This implies any translate of $V$ intersects $\mathbb{F}_q^{d-1}\times \{0\}$ at one point. Hence, for any $\xx \in \Fbb_q^d$, there exists $\xx' \in \Fbb_q^{d-1}\times \{ 0 \}$ such that $(\xx +V) = (\xx ' +V)$, so $(\xx + V) \cap \l( \Fbb_q^{d-1}\times \{ 0 \} \r) = \xx'$. Since $\l( \Fbb_q^{d-1}\times \{ 0\} \r) \subset E$, $ (\xx +V) \cap E \ne \emptyset$ for all $\xx \in \Fbb_q^d$ and all $V \in X_\pi$. That means
\begin{align*}
    \# \lo \mathbf{x}+V \colon (\mathbf{x}+V) \cap  E \ne 0 \ro = q^{d-1} , \quad \text{ for all } V \in X_\pi.
\end{align*}
Thus, $\Tca_2 (X_\pi ,E ,N)=X_\pi$ for all $N\ge q^{d-1}$.
\end{construction}

\begin{construction}\label{c16.2}

   Let $E = \Fbb_q^{d-1}\times I $ where  $ |I| = \left \lfloor \frac{|E|}{q^{d-1}} \right \rfloor $  and $0 \notin I$. For $\lambda =0$, for all $V \in X_\pi$, we have $V \in (0,\cdots ,0 ,1)^\perp = \Fbb_q^{d-1}\times \{ 0 \}$, so $(\xx + V) \cap \l( \Fbb_q^{d-1}\times \{ 0\} \r) = (\xx +V)$  if $ \xx \in \Fbb_q^{d-1} \times \{ 0 \}$ and $\emptyset$ if $\xx \notin \Fbb_q^{d-1}\times \{ 0 \}$. Therefore, $\Fbb_q^{d-1}\times \{ 0\}$ is covered by $\frac{\lv \Fbb_q^{d-1} \times \{ 0 \}\rv}{|V|} = \frac{q^{d-1}}{q} = q^{d-2}$ translations of $V$.
    % Hence, for any $t\in \Fbb_q$, $(\xx +V) \cap \l( \Fbb_q^{d-1}\times \{ t\} \r)$ equivalences with $\xx \in \l( \Fbb_q^{d-1} \times \{ 0 \} \r)$. 
   Hence, by similar argument, for any $t\in \Fbb_q$, $\Fbb_q^{d-1}\times \{ t\}$ is covered by $q^{d-2}$ lines of form $\xx+V$. Note that $E$ is covered by $|I|$ hyperplanes of form $\Fbb_q^{d-1}\times \{ t\} $, so we get
\[ \# \lo \mathbf{x} +V\colon (\mathbf{x}+V)\cap E \ne 0 \ro = q^{d-2}\cdot |I| = \frac{|E|}{q} , \quad \, V \in X_\pi . \]
Then,   
\[ \lv \Tca_2 (X_\pi ,E ,N)\rv = \lv X_\pi \rv , \quad \text{ for all } \,  \frac{|E|}{q}\le N \le q^{d-1}. \] 
From this and note that $\lv X_\pi\rv \sim q^{d-2}$ (by Lemma \ref{Lem1.6} below), one has
\[ \lv \Tca_2 (X_\pi ,E ,N)\rv = \lv X_\pi \rv \sim q^{d-2} \sim \frac{Nq^{d-1}}{|E|}, \]
when $N \sim \frac{|E|}{q}.$
\end{construction}
\section{Open questions}
Motivated by the local-to-global principle, we formulate the following questions.

\paragraph{Distribution of distances:}
Given a set $E$ in $\mathbb{F}_q^d$ and $t\in \mathbb{F}_q^*$. Iosevich and Rudnev \cite{IR} proved that if $t\ne 0$, then the number of pairs $(\mathbf{x}, \mathbf{y})\in E\times E$ such that $||\mathbf{x}-\mathbf{y}||=t$ is at most $\frac{|E|^2}{q}+4q^{\frac{d-1}{2}}|E|$. Notice that $||\mathbf{x}-\mathbf{y}||=t$ is equivalent to $\mathbf{x}-\mathbf{y}\in S^{d-1}$. Thus, from the cover $S^{d-1}=\bigcup_{\lambda \in \mathbb{F}_q} (S^{d-1}\cap \{x_d=\lambda\})$. It would be interesting to look at this question from the local (slide-wise) perspective.

\paragraph{Restricted families of projections with Salem sets:} The authors of \cite{F.R.25} extended Chen's theorem to the setting of $(u, s)$-Salem sets, thus it is natural to ask for analogs of Theorems \ref{p11}--\ref{thm1.7} in this new setting. Note that the $(u, s)$-Salem definition is quite general, and for any set $E$, we can find appropriate $u$ and $s$ such that it is $(u, s)$-Salem. More discussions on this framework can be found in recent papers 
\cite{CKPTZ25, fraser2, JF1, fraser22, fraser33}.
\section{Acknowledgments}
Thang Pham and Doowon Koh gratefully acknowledge the hospitality of the Alfréd Rényi Institute of Mathematics during the summer of 2024, when part of this work was completed. Thang Pham and Do Trong Hoang also thank the Vietnam Institute for Advanced Study in Mathematics (VIASM) for its hospitality, where some initial proofs were developed.

\vspace{1cm}

Vietnam Institute of Educational Sciences. 

Email: hamlq2022@gmail.com

\vspace{0.5cm}
Faculty of Mathematics and Informatics, Hanoi University of Science and Technology.

Email: hoang.dotrong@hust.edu.vn

\vspace{0.5cm}
Faculty of Mathematics and Informatics, Hanoi University of Science and Technology. 

Email: hung.lequang@hust.edu.vn

\vspace{0.5cm}
Chungbuk National University, Department of Mathematics, Chungdae-ro 1, Seowon-Gu, Cheongju City, Chungbuk 28644, Korea

Email: koh131@chungbuk.ac.kr

\vspace{0.5cm}
Institute of Mathematics and Interdisciplinary Sciences at Xidian University, China. 

Email: thangpham.math@gmail.com

\end{document}